\documentclass[bj,reqno]{imsart}
\pdfoutput=1
\setattribute{journal}{name}{}
\usepackage{a4wide}
\usepackage{amsmath,amsfonts}
\usepackage{amsthm}
\usepackage{enumerate}
\usepackage{dsfont}
\usepackage[colorlinks=true,urlcolor=blue,citecolor=blue]{hyperref}
\usepackage[usenames,dvipsnames]{color}
\usepackage{mathrsfs}
\usepackage{placeins}
\usepackage{graphicx}
\usepackage{tikz}
\usetikzlibrary{arrows}
\usepackage{url}
\usepackage{comment}
\usepackage{multicol}
\usepackage{cleveref}
\usepackage{amssymb}
\makeatletter
\DeclareRobustCommand\bigop{%
\mathop{\circledcirc\hspace*{-0.4cm}\vphantom{\sum}\mathpalette\bigop@{\sum}}\slimits@
}
\newcommand{\bigop@}[2]{%
  \vcenter{%
    \sbox\z@{$#1\sum$}%
\hbox{\resizebox{\ifx#1\displaystyle.9\fi\dimexpr\ht\z@+\dp\z@}{!}{$\m@th#2$}}%
  }%
}
\makeatother

\usepackage{calc}
\usepackage{tikz}
\usetikzlibrary{decorations.markings}
\tikzstyle{vertex}=[circle, draw, inner sep=0pt, minimum size=6pt]
\newcommand{\vertex}{\node[vertex]}

\usepackage{natbib}

\numberwithin{equation}{section}

\theoremstyle{plain} 
\newtheorem{theorem}{Theorem}[section]
\newtheorem{lemma}[theorem]{Lemma}
\newtheorem{proposition}[theorem]{Proposition}
\newtheorem{corollary}[theorem]{Corollary}

\theoremstyle{definition}
\newtheorem{definition}[theorem]{Definition}
\newtheorem{dfn}[theorem]{Definition}
\newtheorem{example}[theorem]{Example}

\newtheorem{remark}{Remark}

\newcommand{\bthe}{\begin{theorem}}
\newcommand{\ethe}{\end{theorem}}

\newcommand{\ben}{\begin{enumerate}}
\newcommand{\een}{\end{enumerate}}

\newcommand{\bit}{\begin{itemize}}
\newcommand{\eit}{\end{itemize}}

\newcommand{\beq}{\begin{equation}}
\newcommand{\eeq}{\end{equation}}

\newcommand{\ble}{\begin{lemma}}
\newcommand{\ele}{\end{lemma}}

\newcommand{\bde}{\begin{definition}\rm}
\newcommand{\ede}{\halmos\end{definition}}

\newcommand{\bco}{\begin{corollary}}
\newcommand{\eco}{\end{corollary}}

\newcommand{\bpr}{\begin{proposition}}
\newcommand{\epr}{\end{proposition}}

\newcommand{\brem}{\begin{remark}\rm}
\newcommand{\erem}{\end{remark}}

\newcommand{\bproof}{\begin{proof}}
\newcommand{\eproof}{\end{proof}}

\newcommand{\bexam}{\begin{example}\rm}
\newcommand{\eexam}{\halmos\end{example}}

\newcommand{\bfi}{\begin{fig}}
\newcommand{\efi}{\end{fig}}

\newcommand{\btab}{\begin{tab}}
\newcommand{\etab}{\end{tab}}

\newcommand{\beao}{\begin{eqnarray*}}
\newcommand{\eeao}{\end{eqnarray*}\noindent}

\newcommand{\balo}{\begin{align*}}
\newcommand{\ealo}{\end{align*}}

\newcommand{\balm}{\begin{align}}
\newcommand{\ealm}{\end{align}\noindent}

\newcommand{\beam}{\begin{eqnarray}}
\newcommand{\eeam}{\end{eqnarray}\noindent}

\newcommand{\barr}{\begin{array}}
\newcommand{\earr}{\end{array}}

\newcommand{\E}{\mathbb{E}}
\newcommand{\M}{\mathbb{M}}
\newcommand{\N}{\mathbb{N}}
\renewcommand\P{\mathbf{P}}
\newcommand{\R}{\mathbb{R}}
\def\C{\mathbb{C}}
\def\CA{\mathbb{CA}}

\def\bE{\mathbf E}

\def\MRV{\mathcal{MRV}}

\def\RV{\mathcal{RV}}

\def\binfty{\boldsymbol \infty}

\def\bCA{\mathbb{CA}}

\def\bzero{\boldsymbol 0}
\def\bone{\boldsymbol 1}
\def\bA{\boldsymbol A}

\def\bV{\boldsymbol V}
\def\bX{\boldsymbol X}

\def\bZ{\boldsymbol Z}
\def\binfty{\boldsymbol \infty}

\def\bv{\boldsymbol v}
\def\bx{\boldsymbol x}
\def\by{\boldsymbol y}
\def\bz{\boldsymbol z}
\def\ba{\boldsymbol a}

\def\be{\boldsymbol e}

\newcommand{\tto}{{t\to\infty}}

\newcommand{\al}{{\alpha}}

\newcommand{\halmos}{\quad\hfill\mbox{$\Box$}}

\newcommand{\vague}{\stackrel{\lower0.2ex\hbox{$\scriptscriptstyle
                    \it{v} $}}{\rightarrow}}
\newcommand{\weak}{\stackrel{\lower0.2ex\hbox{$\scriptscriptstyle
                    \it{w} $}}{\rightarrow}}
\newcommand{\what}{\stackrel{\lower0.2ex\hbox{$\scriptscriptstyle
                    \it{\hat{w}} $}}{\rightarrow}}
\newcommand{\eqdis}{\stackrel{\lower0.2ex\hbox{$\scriptscriptstyle
                    \mathrm{d}$}}{=}}
\newcommand{\distr}{\stackrel{\lower0.2ex\hbox{$\scriptscriptstyle
                    \it{d} $}}{\rightarrow}}
\allowdisplaybreaks 
\definecolor{darkred}{RGB}{139,0,0}
\definecolor{darkgreen}{RGB}{0,139,0}
\newcommand{\DAS}[1]{{\color{darkred} #1}}

\arxiv{math.PR/0000000}

\begin{document}

\begin{frontmatter}
\title{Tail probabilities of random linear functions of regularly varying random vectors}
\runtitle{Linear functions of regularly varying vectors}

\begin{aug}
  \author{\fnms{Bikramjit}  \snm{Das} \thanksref{t1}\ead[label=e1]{bikram@sutd.edu.sg}}
 \address{Engineering Systems and Design\\ Singapore University of Technology and Design\\ 8 Somapah Road\\ Singapore 487372, Singapore\\  \printead{e1}}
 \affiliation{Singapore University of Technology and Design}
  \and
  \author{\fnms{Vicky} \snm{Fasen-Hartmann}\ead[label=e2]{vicky.fasen@kit.edu}}
  \address{Institute for Stochastics\\ Karlsruhe Institute of Technology\\ Englerstrasse 2\\ 76131 Karlsruhe, Germany\\    \printead{e2}}
  \affiliation{Karlsruhe Institute of Technology}
  \and
  \author{\fnms{Claudia} \snm{Kl\"uppelberg} \ead[label=e3]{cklu@ma.tum.edu}}
  \address{Center for Mathematical Sciences\\Technical University of Munich\\ Boltzmanstrasse 3\\ 85748 Garching, Germany\\  \printead{e3}}
 \affiliation{Technical University of Munich}
\thankstext{t1}{B. Das was partially supported by the MOE Academic Research Fund Tier 2 grant MOE2017-T2-2-161 on ``Learning from common connections in social networks".}

  \runauthor{B. Das, V. Fasen, and C. Kl\"uppelberg}
\end{aug}

\begin{abstract}
We provide a new extension of Breiman's Theorem on computing tail probabilities of a product of random variables to a multivariate setting.
In particular, we give a complete characterization of regular variation on cones in $[0,\infty)^d$ under random linear transformations.
This allows us to compute probabilities of a variety of tail events, which classical multivariate regularly varying models would report to be asymptotically negligible.
{We illustrate our findings with  applications to risk assessment in financial systems and reinsurance markets under a bipartite network structure.}


\end{abstract}



\end{frontmatter}

\section{Introduction}\label{sec:intro}

In this article we study the probability of tail events for random linear functions of regularly varying random vectors.
Throughout all random elements are defined on the same probability space $(\Omega,\mathcal{F},\P)$.
Suppose $\bZ$ is a non-negative random vector with {\it{multivariate regularly varying tail distribution}} on $\E_d^{(1)}:=\left[0,\infty\right)^{d}\setminus\{\bzero\}$ with index $-\alpha_1\le 0$, denoted as $\MRV(\alpha_{1},\E^{(1)}_d)$. A precise definition of this notion is given in Section \ref{sec:mrv}. Furthermore, let $\bA$ be a $q\times d$ random matrix independent of $\bZ$. For $\bX= \bA\bZ$, our goal is to  find $\P(\bX\in tC )$ for large values of $t$ and a wide variety of sets $C\subset \left[0,\infty\right)^q$.

 A classical result on the tail behavior of a product of random variables, now known as Breiman's Theorem, states that given independent non-negative random variables $Z$ and $A$, where $Z$ has a univariate regularly varying tail distribution with index $-\alpha \le 0$ and $\bE[A^{\alpha+\delta}]<\infty$ for some $\delta>0$, the tail distribution of $X=AZ$ is also regularly varying with index $-\alpha$. More precisely,
  \begin{align}\label{eq:breiman}
  \P(AZ>t) \sim \bE[A^{\alpha}]\P(Z>t), \quad\text{as} \quad t\to \infty.
  \end{align}
This was stated first in \citet{breiman:1965} for $\alpha\in  [0,1]$ and established for all $\alpha\ge0$ in \citet{cline:samorodnitsky:1994}. The inherent applicability of this result to stochastic recurrence equations and portfolio tail risk computations has lead to a few generalizations in the past decades. A generalization of Breiman's Theorem by relaxing the assumption of independence of  the random variables $A$ and $Z$ to {\it{asymptotic independence}}  was provided in \citet{maulik:resnick:rootzen:2002}. On the other hand, a weakening of the conditions on $A$ such that \eqref{eq:breiman} holds was given in \citet{denisov:zwart:2007}.

 A vector-valued generalization of \eqref{eq:breiman} was obtained in \citet[Proposition A.1]{basrak:davis:mikosch:2002a} where the $d$-dimensional non-negative random vector $\bZ\in \MRV({\alpha_{1}},\E_d^{(1)})$ for $\alpha_1\ge 0$ is independent of a $q\times d$-dimensional  random matrix $\bA$ with $\bE\|\bA\|^{\al_1+\delta} <\infty$ for some $\delta>0$. The result states that in such a case $\bX=\bA\bZ \in \MRV(\alpha_{1},\E_q^{(1)})$ where $\E_{q}^{(1)} = \left[0,\infty\right)^{q}\setminus \{\bzero\}$. A generalization of this result with respect to the dependence and joint regular variation assumptions on $(\bA,\bZ)$ was given in \citet{fougeres:mercadier:2012}. On the other hand,  \citet[Theorem 2.3]{janssen:drees:2016}   generalized  Proposition A.1 in \citet{basrak:davis:mikosch:2002a} so that one may compute probabilities of tail sets $C$ contained in $\E_q^{(q)} = (0,\infty)^q$ when $q=d$ and $\bA$ is of full rank (and certain other conditions).
 {For $\bZ\in \MRV({\alpha_{d}},\E_d^{(d)})$} 
 they show that $\bX=\bA\bZ \in \MRV(\alpha_{d},\E_d^{(d)})$. 

 Consider the following example to fix ideas in this setting. Let $\bZ=(Z_1,\ldots,Z_d)^{\top}$ be comprised of iid (independent and identically distributed) Pareto random variables with $\P(Z_{i}>z)=z^{-\alpha}$ for $z>1$, where $\alpha>0$, and let $\bA$ be a $d\times d$ random matrix independent of $\bZ$ satisfying the conditions for both   \citet[Proposition A.1]{basrak:davis:mikosch:2002a} and  \citet[Theorem 2.3]{janssen:drees:2016}. Then $\bX=\bA\bZ\in\MRV(\alpha,\E_{d}^{(1)})$ and $\bX=\bA\bZ\in\MRV(d\alpha,\E_{d}^{(d)})$. Hence for sets of the form $ [\bzero,\bx]^{c}$ and $(\bx,\infty)$ with $\bx>0$, we are able to compute for $t\to\infty$:
 \begin{align}
 & \P(\bA\bZ \in t  [\bzero,\bx]^{c})  \sim  t^{-\alpha}\, \bE[\mu_{1}(\bA\bZ\in[\bzero,\bx]^{c})], \label{firsteq}\\
 & \P(\bA\bZ \in t  (\bx,\binfty)) \sim  t^{-d\alpha}\,\bE[\mu_{d}(\bA\bZ\in (\bx,\infty))], \label{secondeq}
 \end{align}
for some measures $\mu_{1}$ and $\mu_{d}$ to be elaborated on later.
Moreover, the expectations on the right hand side of both \eqref{firsteq} and \eqref{secondeq} are non-trivial and finite; hence our probability estimates are valid. Thus \eqref{firsteq} allows us to compute probabilities of events described as ``at least one of the components of $\bX$ is large'', whereas  \eqref{secondeq} allows us to compute probabilities of events described as ``all components of $\bX$ are large''. Natural questions to inquire of here would be, what if we want to compute such probabilities when the matrix $\bA$ is not invertible, or perhaps $q\neq d$. We may also wish to find the probability that ``at least three of the components of $\bX$ are large'' or ``exactly two of the components of $\bX$ are large". We can check that, although a probability computation akin to \eqref{firsteq} is possible in such a case, it will often render the measure $\mu_{1}$ and hence, the right hand side of \eqref{firsteq} to be zero. On the other hand, \eqref{secondeq} will fail to answer such a question if either $q\neq d$ or the particular set of concern does not have all components to be large. To the best of our knowledge, \eqref{firsteq} and \eqref{secondeq} are the only results that compute probabilities of extreme sets for random linear functions of regularly varying vectors. In our work, we provide a generalization of Breiman's Theorem which allows us to compute such probabilities for more general extreme sets. For example, in this particular setting of $\bZ$ being iid Pareto, our results show that
  \begin{align}\label{eq:breinew}
 \P(\bA\bZ \in tC) \sim  t^{-i\alpha}\, \bE_i^{(k)}\left[{\mu}_{i}(\bA^{-1}(C))\right], \quad \quad t\to\infty.
   \end{align}
   where the index $i\in \{1,\ldots,d\}$ depends on the structure of the matrix $\bA$ and the set $C$, and $\bE_i^{(k)}$ represents an expectation over an appropriate subspace of the probability space; see Section \ref{subsec:mainBreiman} for the definition. Finding the correct exponent $i$ under a general set-up forms the basis of this paper.

 {\it{Further related literature:}}  A few other publications have also exhibited interesting applications and generalizations of Breiman's Theorem, albeit in different contexts. In \citet{jessen:mikosch:2006}, the authors provide partial converses to Breiman's Theorem: assuming $A$ and $Z$ to be non-negative independent random variables, if $AZ$ has a regularly varying tail distribution, they find conditions when $Z$ will also have a regularly varying tail distribution.   In \citet{tillier:wintenberger:2018} we find an extension of Breiman's multivariate result to vectors of random length, determined for instance by a Poisson random variable. In a more general setting,  \citet{chakraborty:hazra:2018}, extend Breiman's result for multiplicative Boolean convolution of regularly varying measures. Finally, the monograph  \citet{buraczewski:damek:mikosch:2016} provides many applications of Breiman's result  and its generalizations in the area of stochastic modeling with power-law tails.

 Our interest in computation of  probabilities of the form \eqref{eq:breinew} is motivated by a wide range of applications in mind. Regularly varying tail distributions have been used to model power-law tail behavior in stochastic models in applications including hydrology, finance, insurance, telecommunication, social networks and many more. A regularly varying random vector like $\bZ\in \left[0,\infty\right)^{d}$ can be used to represent {investment risks} from multiple stocks (in finance) or losses  pertaining to different insurance companies (in an  insurance context). In such applications a $q\times d$ random matrix $\bA$ represents randomly weighted choices of portfolios $\bZ$ of a group of stockholders or business entities,  or, randomly weighted exposures of insurance companies to losses, respectively.  Thus a common quantity of interest to compute here is $\P(\bA\bZ\in tC)$ for tail sets $C$ representing a variety of worst case scenarios relating to multiple portfolios, or bankruptcy or loss for multiple insurers.

 Our paper is organized as follows. We provide a summary of  notations used in the paper in Section \ref{subsec:notation} to finish up the introduction. In Section \ref{sec:mrv}, we discuss  multivariate regular variation with $\M$-convergence in different subspaces of $\left[0,\infty\right)^{d}$ which provides a set up for the main result of the paper. Our main result extending Breiman's Theorem is developed in Section \ref{sec:Breiman}. In Section \ref{sec:applications}, we provide applications of the model in the context of bipartite networks, where $q$ agents can be exposed to the risk of $d$ objects where $\bZ\in\left[0,\infty\right)^d$ are the risks of the objects. The exposures of the agents is represented by $\bX=\bA\bZ$ and illustrate the behavior of tail risk of the agents for possible structures of the weighted adjacency matrix $\bA\in\left[0,\infty\right)^{q\times d}$. We conclude with indications to future directions of research in Section \ref{sec:concl}.

\subsection{Notations}\label{subsec:notation}
Various notations and concepts used in this paper are summarized in this section.
Vector operations are always understood component-wise, e.g., for vectors $\bx=(x_1,\ldots,x_d)^{\top}$ and $\by=(y_1,\ldots,y_d)^{\top}$, $\bx\le \by$ means $x_i\le y_i$ for all $i$.  For a constant $t\in \R$ and a set $C\subseteq \R^{d}$, we denote by $tC:= \{t\bx: \bx\in C\}$.
Further notations are tabulated below. References are provided wherever applicable.

$$
\begin{array}{ll}
\RV_\beta & \text{Regularly varying functions with index $\beta\in\R$; that
           is, functions $f:\mathbb{R}_+\mapsto \mathbb{R}_+$}\\
{}& \text{satisfying $\lim_{t\to\infty}f(tx)/f(t)=x^\beta,$ for $x>0$; see  \cite{resnickbook:2008};}\\
{}&\text{\cite{bingham:goldie:teugels:1989,dehaan:ferreira:2006}; for details.}\\[2mm]
V\in\RV_{-\al} &    \text{ A random variable {$V$} with distribution function $F_V$ is regularly varying} \\
{}& \text{(at infinity) if $\overline{F}_V:= 1-F_V \in \RV_{-\alpha}$  for some $\alpha\ge 0$.}\\[2mm]
\R_{+}^{d} & [0,\infty)^{d} \text{ for dimension $d\ge 1$.}\\[2mm]
v^{(1)}, \ldots, v^{(d)}   & \text{Order statistics of $\bv=(v_1,\ldots,v_d)^{\top}\in \R_+^d$ such that $v^{(1)}\ge v^{(2)}\ge \ldots \ge v^{(d)}$.}  \\[2mm]
\CA_d^{(i)} & \{\bv\in \R^{d}_+: v^{(i+1)}=0\}, \text{ $i=1,\ldots,d-1$; also define } \CA_d^{(0)} = \{\bzero\}.\\[2mm]
\M(\mathbb{C}\setminus \mathbb{C}_0) & \text{The set of all non-zero
measures on $\mathbb{C}\setminus \mathbb{C}_0$ which are finite on Borel subsets }\\&\text{bounded away from  $\mathbb{C}_0$.}
\\[2mm]
\mu_n\to \mu & \text{Convergence in $\M(\mathbb{C}\setminus
  \mathbb{C}_0)$;} 
  \text{ see Section~\ref{Def:M convergence} and
   \cite{das:mitra:resnick:2013};} \\
   {}& \text{\cite{hult:lindskog:2006a,
  lindskog:resnick:roy:2014} for details.}\\[2mm]
\end{array}
$$

$$
\begin{array}{ll}
\MRV(\alpha,b,\mu,\E) & \text{Multivariate regular variation on the space $\E=\C\setminus\C_0$, where $\C$ and $\C_0$}\\
 & \text{are closed cones in $\R_+^d$. Here $-\alpha \le 0$ is the index of regular variation,}\\
  & \text{$b$ is the scaling function, and $\mu$ is the limit measure. We often omit one}\\
  & \text{or more of the arguments. See Definition \ref{def:mrv} for details.}\\[2mm]
 \tau_{q}^{(k)}(\bx) &  d(\bx,\CA^{(k-1)}_{q}) = x^{(k)}, \text{ the distance between $\bx\in \R_+^q$ and $\CA_q^{(k-1)}$,}  \\
& \text{where }  d(\bx,\by)=||\bx-\by||_\infty; \text{ see Section \ref{sec:Breiman} for details.}
\\[2mm]
   \tau_{d,q}^{(k,i)}(\bA) & \sup\limits_{\bz \in \E_d^{(i)}} \frac{\tau_{q}^{(k)} (\bA \bz)}{\tau_{d}^{(i)}(\bz)}  \text{ for  $\bA\in\R_+^{q\times d}$ ; see Section \ref{sec:Breiman} for details.}\\[2mm]
\|\cdot\| & \text{For $\bx\in\R^d$,  $\|\bx\|$ denotes some vector norm and for a matrix $\bA\in\R^{q\times d}$,  }\\
& \text{$\|\bA\|$ denotes the corresponding operator norm.}
\end{array}
$$

\section{Multivariate regular variation and convergence concepts}\label{sec:mrv}

We use the notion of $\M$-convergence of measures to define multivariate regular variation on Euclidean spaces and subsets thereof; see \citet{das:mitra:resnick:2013,lindskog:resnick:roy:2014} for details.
In particular, we investigate regular variation of a random vector $\bX$, which is given as $\bX=\bA\bZ$, where $\bZ\in \R_+^d=\left[0,\infty\right)^d$ is multivariate regularly varying with index $-\al\le 0$ and $\bA$ is a $q\times d$ random matrix independent of $\bZ$ such that $\bE[\|\bA\|^{\al+\delta}]<\infty$ for some $\delta>0$ and some operator norm $\|\cdot\|$ for matrices.

Our goal is to obtain a complete picture concerning linear functions $\bX=\bA\bZ$ which possess multivariate regular variation on a sequence of subspaces of $\R_+^d$ (also called hidden regular variation),
thus extending results from \citet{basrak:davis:mikosch:2002a,janssen:drees:2016}.
The particular choice of subsets where we seek regular variation are natural, depending on the type of extreme sets for which we seek to find probabilities; see \citet{mitra:resnick:2011hrv} for examples.
The necessary definitions and results formulated with respect to $\M$-convergence are discussed below.

Consider the space $\R_{+}^{d}$ endowed with a metric $d(\bx,\by)$ satisfying for some $c>0$
\begin{equation}\label{eq:Metricscale}
d(c\bx, c\by)=c\,d(\bx,\by), \qquad   (\bx,\by) \in \R_+^d \times \R_+^d.
\end{equation}
Any metric $d$  defined by a norm as $d(\bx,\by)=\|\bx-\by\|$ will always satisfy \eqref{eq:Metricscale}.  {In this paper, we use the sup-norm $d(\bx,\by)=||\bx-\by||_\infty$ as our choice of metric $d$, since the distance of a point $\by \in \R_+^d$ to a specific closed set can be represented as an order statistic of the co-ordinates of $\by$; see \eqref{eq:taukq}.}

Recall that a cone $\mathbb{C} \subset \R_+^d $ is a set which is closed under
scalar multiplication: if $\bx \in\mathbb{C}$ then $c\bx \in \mathbb{C}$ for $c>0$.
A closed cone of course, is a cone which is a closed set in $\R_+^d$.
Now we define multivariate regular variation using convergence of measures
on a closed cone $\mathbb{C}\subset {\R^d_+}$ with a closed cone  $\mathbb{C}_0 \subset  \mathbb{C}$ deleted.
Moreover, we say that a  subset $\Lambda \subset \mathbb{C} \setminus \mathbb{C}_0$ is {\it
  bounded away from\/} $\mathbb{C}_0$ if $d(\Lambda,\mathbb{C}_0)=\inf\{d(\bx,\by):\bx\in\Lambda,\by\in\C_0 \}>0$. The class of Borel measures on $\C\setminus\C_0$ that assign finite measure to all Borel sets $B\subset \C\setminus\C_0$, which are bounded away from $\C_0$, is denoted by $\M(\mathbb{C}\setminus \mathbb{C}_0)$.

{In this paper, regular variation on cones is defined using $\M$-convergence, which is slightly different from vague convergence which has been traditionally used in multivariate regular variation. Reasons for the preference of $\M$-convergence are presented in \citet[Remark~1.1]{das:resnick:2015}; see also \citet{das:mitra:resnick:2013, lindskog:resnick:roy:2014}.
In the space $\E_d^{(1)}=\R^d_+\setminus\{\bzero\}$ the notions of vague convergence and $\M$-convergence are identical.

\begin{definition} \label{Def:M convergence}
Let $\C_0\subset \C\subset \R_+^d$ be closed cones containing $\bzero$.
 Let $\mu_n,\mu$ be Borel measures on $\M(\mathbb{C}\setminus \mathbb{C}_0)$ and $\int f\,d\mu_n\to\int f\,d\mu$ as $n\to\infty$ for any bounded, continuous, real-valued function $f$ whose support is bounded away from $\C_0$, then we say {\em $\mu_n$ converges to $\mu$ in $\M(\mathbb{C}\setminus \mathbb{C}_0)$}, and write $\mu_n\to\mu$ in $\M(\C\backslash\C_0)$.
\end{definition}

\begin{dfn}
\label{def:mrv}
Let $\C_0\subset \C\subset \R_+^d$ be closed cones containing $\bzero$.
A random vector $\bV=(V_1,\ldots,V_d)^{\top} \in \C$ is {\it{regularly varying}} on $\mathbb{C}
\setminus \mathbb{C}_0$ if there exists a  function
$b(\cdot) \in \RV_{1/\alpha}$ for {$\alpha \ge 0$}, called the {\it scaling
  function\/}, and a non-null (Borel) measure $\mu(\cdot)
\in \M(\mathbb{C}
\setminus \mathbb{C}_0)$ called the {\it limit or tail measure\/}
{such that}
\begin{equation*}
t\P(\bV/b(t) \in \,\cdot \,) \to \mu(\cdot), \quad  t \to\infty,
\end{equation*}
in $\M(\mathbb{C}\setminus \mathbb{C}_0)$.
We write $\bV \in \MRV(\alpha, b, \mu, \mathbb{C}\setminus \mathbb{C}_0)$ or,  $\bV \in \MRV(\alpha,  \mu, \mathbb{C}\setminus \mathbb{C}_0)$ if the scaling function is contextually irrelevant. If  $\mathbb{C}\setminus \mathbb{C}_0= \left[0,\infty\right)^d\setminus\{\bzero\}=:\E_{d}^{{(1)}}$, we simply write  $\bV \in \MRV(\alpha, \mu,\E_{d}^{(1)})$ or $\bV \in \MRV(\alpha, \mu)$.
\end{dfn}
For $\alpha>0$, a possible choice of $b$ is given by using $t\P(\max\{V_1,\dots,V_d\}>b(t))\to 1$ as $t\to\infty$.
Since $b \in \RV_{1/\alpha}$, the limit measure $\mu(\cdot) $ has a scaling property:
\begin{equation*}
\mu(c \;\cdot\,) =c^{-\alpha} \mu(\,\cdot\,),\qquad c>0.
\end{equation*}




\subsection{Regular variation on a sequence of subspaces}\label{subsec:hrv}

We define regular variation on a specific sequence of subspaces of $\R^{d}_{+}$ following \citet{mitra:resnick:2011hrv}.
For $\bv\in \R^{d}_+$, write $\bv= (v_{1}, \ldots,v_{d})^{\top}$.
Moreover, the order statistics for any vector $\bv\in\R^d_+$ is defined as
\beao
v^{(1)}\ge v^{(2)}\ge \ldots \ge v^{(d)},
\eeao
where $v^{(i)}$ denotes the $i$-th largest component of $\bv$.
First we define closed sets which we think of as a union of co-ordinate hyper-planes of various dimensions in $\R^d_+$. Let $\CA_d^{(0)} := \{\bzero\}$ and for $1\le i \le d-1$ define
\begin{align*}
    \CA_d^{(i)} & = \bigcup_{1\le j_{1}<\ldots<j_{d-i} \le d}  \{\bv\in \R^{d}_+: v_{j_{1}}=0,\ldots, v_{j_{d-i}}=0\}
                 := \{\bv\in \R^{d}_+: v^{(i+1)}=0\}.
\end{align*}
Here $\bCA^{(i)}_{d}$ represents the union of all $i$-dimensional co-ordinate hyperplanes in $\R^{d}_{+}$.  Also define $\CA_d^{(d)}:=\{\bv\in \R^d_+: v^{(d)}>0\}$.
Now define the following sequence of subcones of $\R^d_+$: 
\begin{align}
\E^{(1)}_{d}  := &\;\R^{d}_+ \setminus \CA_d^{(0)} =  \R_+^{d} \setminus \{\bzero\} =   \{\bv\in \R^d_+: v^{(1)}>0\}, \label{def:ed1}\\
\E^{(i)}_{d}  := &\;  \R^d_+ \setminus \bCA^{(i-1)}_{d} = \; \{\bv\in \R^d_+: v^{(i)}>0\}, \quad\quad 2\le i \le d. \label{def:edd}
\end{align}
  Hence $\E^{(1)}_{d}$ is the non-negative orthant with $\{\bzero\}=\bCA^{(0)}_{d}$ removed, $\E^{(2)}_{d}$ is the non-negative orthant with all one-dimensional co-ordinate axes removed, $\E^{(3)}_{d}$ is the non-negative orthant with all two-dimensional co-ordinate hyperplanes removed, and so on.
 Clearly, we have
 \beao
\E^{(1)}_{d} \supset \E^{(2)}_{d} \supset \ldots \supset \E^{(d)}_{d}.
\eeao
Note that according to our definition $\E_d^{(d)}=\CA_d^{(d)}.$
We also define for $i=1,\ldots, d$,
 \begin{align}\label{eq:cminusc}
  \bCA_d^{(i)}\backslash \bCA_d^{(i-1)}  & = \{\bv\in \R_+^d: \text{exactly $i$ co-ordinates of $\bv$ are positive}\}
                      =: \bigcup_{j=1}^{\binom{d}{i}} \widetilde \bCA_d^{(i)}(j),
 \end{align}
 where $\widetilde\bCA_d^{(i)} (j)$ denotes the $j$-th $i$-dimensional co-ordinate hyperplane in $\R_+^d$ with $i$ positive and $d-i$ zero co-ordinates in some ordering of the hyperplanes.
 We note in passing that
 \[\CA_d^{(d)} =\E_d^{(d)}= \widetilde{\CA}_d^{(d)}(1).\]

 A recipe for finding regular variation in the above sequence of cones can be devised as follows.
 To start with, suppose $\bV \in \MRV(\alpha_1,b_1,\mu_1, \E^{(1)}_d)$ with $\alpha_1>0$.
\begin{enumerate}[(1)]
\item
 If $\mu_1(\E^{(d)}_d)>0$, we seek no further regular variation on cones of $\R_{+}^{d}$.
\item
If $\mu_1(\E^{(d)}_d)=0$, we may find an $i\in\{2,\ldots,d\}$ such that $\mu_1(\E^{(i-1)}_d)>0$, yet $\mu_1(\E^{(i)}_d)=0$.
Hence $\mu_1$ concentrates on $\bCA^{(i-1)}_d$.
So we seek regular variation in $\E^{(i)}_d = \R_+^d\setminus\bCA^{(i-1)}_d$.
Suppose there exists $b_{i}(t)\uparrow \infty$ with $\lim_{t\to\infty} b_{1}(t)/b_i(t) = \infty$ and $\mu_{i}\neq0$ on $\E^{(i)}_d$    such that $\bV \in \MRV(\alpha_{i},b_{i},\mu_{i}, \E^{(i)}_d)$.
Then,  $\alpha_i\geq\alpha_1$, $b_{i}(\cdot) \in \RV_{1/\alpha_{i}}$ and $\mu_{i}(c\;\cdot\,) = c^{-\alpha_{i}} \mu_{i}(\,\cdot\,)$ for $c>0$. Hence $\bV$ has regular variation on $\E_d^{(i)}$ with parameter $-\alpha_i$.
\item  In the next step, if  $\mu_{i}(\E^{(d)}_d)>0$, we stop looking for regular variation; otherwise we keep seeking regular variation through $\E^{(i+1)}_d, \ldots, \E^{(d)}_d$ sequentially.
\end{enumerate}
The idea of regular variation on a sequence of cones is easier understood with an example.

\begin{example}\label{ex:indpar}
For $d\ge 2$, suppose $\bV=(V_{1}, \ldots, V_{d})^\top$ and $V_1,\ldots,V_d$ are iid Pareto($\alpha$) random variables with $\alpha>0$ such that $\P(V_i>t)= t^{-\al}$, $t\ge 1$.  
\begin{enumerate}[(i)]
\item First we observe that for all $i=1 ,\ldots,d$, we have $\bV\in\MRV(\alpha_i,b_i,\mu_i, \E^{(i)}_d)$ with $\alpha_i=i\alpha$ and $b_i(t) = t^{1/(i\alpha )}$  where
the limit measure $\mu_i$ on $\E^{(i)}_d$ is such that for any $\bz=(z_{1},\ldots,z_{d})^{\top} \in \E^{(d)}_d$,
 \begin{align}
 \mu_i(\{\bv\in \E^{(i)}_d: v_{j_{1}} > z_{j_{1}}, \dots, v_{j_{i}} > z_{j_{i}} &\text{ for some }1\le  j_{1} <\ldots< j_{i}\le d\}) \nonumber \\ & =  \frac{1}{\binom{d}{i}}\sum_{1\le j_{1}< \cdots <j_{i}\le d} \left({z_{j_{1}}z_{j_2}\cdots z_{j_{i}}}\right)^{-\alpha}. \label{ex:muidef}
 \end{align}
 This follows from Example~5.1 in \citet{maulik:resnick:2005} and Example~2.2 in \citet{mitra:resnick:2011hrv}.  Hence, if
 \begin{align}\label{ex:setC}
 C= \{\bv\in\R_+^d: v_{{1}} > z_{{1}}, \dots, v_{{i}} > z_{{i}} \}
 \end{align}  we find
 \begin{align}\label{ex:Vcoord}
\P(\bV\in t C) &=  t^{-i\alpha} \left({z_{{1}}z_{2}\cdots z_{{i}}}\right)^{-\alpha} +o(t^{-i\alpha}) ,\quad t\to\infty.
\end{align}
 \item The measure $\mu_i$ as defined in \eqref{ex:muidef} concentrates on $\CA_d^{(i)}\setminus\CA_d^{(i-1)}$.
\item In general, from part (i) we conclude that  for any Borel set $C\subset \E_d^{(i)}$ which is bounded away from $\CA_d^{(i-1)}$,
 \begin{align*}
\P(\bV\in t C) & = t^{-i\alpha} \mu_i(C) + o(t^{-i\alpha}), \quad\quad t\to\infty.
\end{align*}
So, in case $\mu_i(C)=0$, we get $\P(\bV\in t C) = o(t^{-i\alpha})$ as $t\to\infty$.
However, if  $C$ is of the form \eqref{ex:setC}, or a finite union of such sets (for fixed $i$), from \eqref{ex:Vcoord} we know that $\mu_i(C)>0$.
\end{enumerate}
\end{example}

\begin{remark}
Although multivariate regular variation can be defined for a very general class of cones  in $\R^{d}_+$  (see \citet{das:mitra:resnick:2013, lindskog:resnick:roy:2014, mitra:resnick:2011hrv} for examples), for the purposes of this paper, restricting to the  sub-cones $\E_d^{(1)},\ldots,\E_d^{(d)}$ defined in \eqref{def:ed1} and \eqref{def:edd} suffices. For an example of  regular variation with infinite sequence of indices on an infinite sequence of cones contained in the space $\R_+^2$, see \citet[Example 5.3]{das:mitra:resnick:2013}.
\end{remark}

\begin{definition}\label{def:bit}
Suppose $\bV=(V_1,\ldots,V_d)^{\top}\in\MRV(\alpha_i,b_i,\mu_i, \E^{(i)}_d)$ and $F_{V^{(i)}}^{\leftarrow}(s)=\inf\{t\in\R: F_{V^{(i)}}(t)\geq s\}$ is the generalized
inverse of the distribution function $F_{V^{(i)}}$ of $V^{(i)}$, where $V^{(1)}\geq \ldots\geq V^{(d)}$ are the order statistics of $V_1,\ldots,V_d$.
If $b_i(t)=F_{V^{(i)}}^{\leftarrow}\left(1-1/t\right)$, we call $b_i(\cdot)$ the {\em canonical choice} of the scaling function.
\end{definition}

\section{Breiman's Theorem and  regular variation on Euclidean subspaces}\label{sec:Breiman}

In this section we provide a complete characterization of  the vector-valued generalization addressed in   \citet[Proposition A.1]{basrak:davis:mikosch:2002a} for the space $\E_q^{(1)}$ and its subsequent {modification for} $\E_q^{(q)}$ for $q=d$  provided in \citet[Theorem 2.3]{janssen:drees:2016}.  We investigate the vector $\bX=\bA\bZ$, where $\bA \in \R^{q\times d}$ is a random matrix which is independent of $\bZ\in \R_+^d$, and $\bZ$ is multivariate regularly varying on subspaces $\E_d^{(i)}$ for $i=1,\ldots, d$. We provide asymptotic rates of convergence  of tail probabilities for  $\P(\bA\bZ\in t C)$ for Borel sets $C\subset\E_q^{(k)}$ for $k= 1,\ldots,q$.  For the sake of convenience, first we present the  two available results addressing this issue.

Throughout $\|\cdot\|$ denotes an arbitrary vector and operator norm, only the metric $d(\cdot,\cdot)$ is always defined by the sup-norm.
Most results quoted from previous papers appeared with asymptotic properties and definitions in terms of vague convergence, we restate them here with respect to $\M$-convergence.


  \begin{theorem}[{\citet[Proposition A.1]{basrak:davis:mikosch:2002a}}]\label{th31}
 Let $\bZ\in \R_+^d$ be a random vector such that  $\bZ \in \MRV(\al_{1}, \mu_1,\E_d^{(1)})$ with $\alpha_1\ge 0$ and $\bA \in \R^{q \times d}_{+}$  be a random matrix independent of $\bZ$ with  $0<\bE[\|\bA\|^{\alpha_1+\delta}] <\infty$ for some $\delta>0$. Then
  \begin{align}\label{eq:multBreiman}
\frac{\P(t^{-1}\bA\bZ \in \,\cdot\,)}{\P(\|\bZ\|>t)} \to  \bE\left[ \mu_1(\{\bz\in \E^{(1)}_{d}:\bA\bz \in \,\cdot\,\})\right]=:\overline{\mu}_{1}(\cdot),\quad \tto,
  \end{align}
 in $\M(\E^{(1)}_{q})$.  In particular, we have $\bA\bZ\in \MRV(\alpha_{1},\overline{\mu}_{1},\E_{q}^{(1)})$.
  \end{theorem}

  \begin{remark} A couple of remarks are in order here.
  \begin{enumerate}[(i)]
      \item For $\|\bZ\|$ to become large, it suffices that one component of $\bZ$ becomes large. Hence $\P(\|\bZ\|>t) \sim c\,\P(Z^{(1)}>t)$ as $t\to\infty$ for some constant $c>0$, {and $\P(Z^{(1)}>t)$ provides the rate of convergence of $\P(t^{-1}\bA\bZ \in \,\cdot\,)$ to zero.}
\item The observation in \eqref{firsteq} is an easy consequence of this theorem.
  \end{enumerate}
  \end{remark}

For certain sets in $\E_q^{(1)}$, it is possible that the right hand side of \eqref{eq:multBreiman} turns out to be zero, rendering the result uninformative. A partial solution for guaranteeing a non-zero limit in \eqref{eq:multBreiman}
is provided in \citet{janssen:drees:2016}, when $q=d$ and where convergence occurs in the space $\E^{(d)}_{d}=(0,\infty)^{d}$, which means that we focus on sets, where all components of $\bX=\bA\bZ$ are large, translated into the event $ \{X^{(d)}>t\}$.


The formal setting in \citet{janssen:drees:2016} is as follows. Define $\tau: \R^{d}_{+}\to \R_{+}$ to be the distance of a point $\bx \in \R^{d}_{+}$ from the space $\bCA^{(d-1)}_{d}: = \left[0,\infty\right)^{d} \setminus (0,\infty)^{d}$ in the sup-norm, given by $\tau(\bz) : = d(\bz, \bCA_{d}^{(d-1)}) = z^{(d)}$.
For a deterministic matrix $\bA \in \R^{d\times d}$ we define the analog
 \beam\label{mnorm}
 \tau(\bA):= \sup_{\bz\in \E^{(d)}_{d}: \tau(\bz)=1} \tau(\bA \bz)
 = \sup_{\bz\in \partial\aleph^{(d)}_d} \tau(\bA \bz),
 \eeam
where for $i\in \{1,\ldots,d\}$,
\begin{align*}
\partial\aleph_{d}^{(i)} \, = \, \{\bx \in \E_d^{(i)} : d(\bx, \CA_d^{(i-1)}) = 1\}.
\end{align*}
\begin{theorem}[{\citet[Theorem 2.3]{janssen:drees:2016}}]\label{lem:jd2017thm23}
  Let $\bZ\in \R_+^d$ be a random vector such that $\bZ \in \MRV(\al_{d},\mu_d,\E^{(d)}_{d})$ and $\bA \in \R^{d \times d}$ be a  random matrix independent of $\bZ$. Assume $\tau(\bA)>0$ almost surely and  $\bE[\tau(\bA)^{\alpha_d+\delta}] <\infty$ for some $\delta>0$. Then
   \begin{align*}
   \frac{\P(t^{-1}\bA\bZ \in \cdot)}{\P(\tau(\bZ)>t)} \to \bE \left[\mu_d(\{\bz\in\E_d^{(d)}: \bA\bz \in \cdot\})\right]=: \overline{\mu}_{d}(\cdot),\quad \tto,
   \end{align*}
 in $\M(\E^{(d)}_{d})$. In particular, we have $\bA\bZ\in \MRV(\alpha_{d},\overline{\mu}_{d},\E_{d}^{(d)})$.
\end{theorem}

 \begin{remark}\label{rem:jandrees}
 A couple of remarks are necessary to explain the result stated above.
  \begin{enumerate}[(i)]
   \item  Note that  $\P(\tau(\bZ)>t)= \P(Z^{(d)}>t)$ which provides the rate of convergence of $\P(t^{-1}\bA\bZ\in\cdot)$ to zero {as $t\to\infty$}.
   \item The observation in \eqref{secondeq} is an easy consequence of Theorem \ref{lem:jd2017thm23}.
   \item    Theorem \ref{lem:jd2017thm23} is designed for a specific situation in the context of stochastic volatility models. It is restrictive in its assumptions and may fail to capture a variety of instances where the right hand side of \eqref{eq:multBreiman} is zero.
In particular, for a square random matrix $\bA$ with almost surely non-negative entries, Theorem \ref{lem:jd2017thm23}  requires  that $\bA$ is almost surely invertible and, moreover, that its inverse has almost surely non-negative entries (see \citet[Lemma 2.2]{janssen:drees:2016}).
This entails that almost all realizations of $\bA$ are row permutations of diagonal matrices with positive diagonal entries (cf. \citet{dingrhee}).
  \end{enumerate}
  \end{remark}

 \subsection{Extension of Breiman's Theorem to Euclidean subspaces}

 In light of the previous results,
 we provide a multivariate extension to Breiman's Theorem which entails non-trivial convergence for a multitude of forms of $\bA$. Let $\bA\in \R_{+}^{q\times d}$ be deterministic.
We define the analog sequence of subcones of $\R_+^q$ as in \eqref{def:ed1}-\eqref{def:edd}  and proceed as follows. For  $k =1,\dots, q$, define $\tau_q^{(k)}:\R_+^q\to\R_+$ to be the distance of a point $\bx \in \R^{q}_{+}$ from the space $\CA^{(k-1)}_{q}$ in the sup-norm, given by
\begin{align}\label{eq:taukq}
\tau_{q}^{(k)}(\bx) & = d(\bx,\CA^{(k-1)}_{q}) = x^{(k)}.
\end{align}

Furthermore, we define in analogy to \eqref{mnorm} the function $\tau_{q,d}^{(k,i)}:\R_+^{q\times d}\to\R_+$ given by
\beam\label{mtau}
\tau_{q,d}^{(k,i)} (\bA) =\sup_{\bz \in \E_d^{(i)}} \frac{\tau_{q}^{(k)} (\bA \bz)}{\tau_{d}^{(i)}(\bz)} =\sup_{\bz\in\E_d^{(i)}} \frac{(\bA\bz)^{(k)}}{z^{(i)}} = \sup_{\bz\in \partial \aleph_d^{(i)}} \tau_{q}^{(k)} (\bA \bz).
\eeam
Note that  $\tau_{q,d}^{(q,d)} (\bA) =\tau(\bA)$ from \eqref{mnorm} if $q=d$.

Although the functions $\tau_{q}^{(k)}, \tau_{q,d}^{(k,i)}$ are not necessarily seminorms on the induced vector space (see \citet[Section~5.1]{horn:johnson:2013}), they admit to some useful  properties as  listed below.
We call a row of $\bA$ {\it{trivial}}, if it is a zero vector.

   \begin{lemma}\label{lem:tauqi}
    For every deterministic matrix $\bA \in \R^{q\times d}_{+}$ and $\bz \in \R^{d}_{+}$ the following hold
    for $i=1,\ldots,d$ and $k=1,\ldots,q$:
    \begin{enumerate}
   \item[(a)] \, $\tau_{q}^{(k)}(\bA\bz) \le \tau_{q,d}^{(k,i)}(\bA)\tau_{d}^{(i)}(\bz)$.
    \item[(b)] \, $\tau_{q,d}^{(k,i)}(\bA)\leq \tau_{q,d}^{(k-1,i)}(\bA)$.
    \item[(c)] \, $\tau_{q,d}^{(k,i)}(\bA)\leq \tau_{q,d}^{(k,i+1)}(\bA)$.
     \item[(d)] \,
     $\tau_{q,d}^{(q,1)}(\bA)>0$ if and only if all rows of $\bA$ are non-trivial.
    \item[(e)] \, $\tau_{q,d}^{(k,1)}(\bA)\leq \tau_{q,d}^{(1,1)}(\bA)<\infty$.
      \end{enumerate}
    \end{lemma}
\begin{proof} $\mbox{}$
\begin{enumerate}
\item[(a)] By definition we have
\beao
\tau_{q}^{(k)}(\bA\bz) = \tau_{q}^{(k)}\Big(\bA\frac{\bz}{\tau_{d}^{(i)}(\bz)}\Big)\tau_{d}^{(i)}(\bz)\le \tau_{q,d}^{(k,i)}(\bA)\tau_{d}^{(i)}(\bz).
\eeao
\item[(b)] and (c) immediately follow from the definition.
\item[(d)] If $\bA = (A_{ij})_{i,j}$ has {no trivial row}, denoting $\be = (1,\ldots,1)^{\top}\in\R_+^d$, we have     $$\tau_{q,d}^{(q,1)} (\bA) \ge \frac{\tau_q^{(q)} (\bA\be)}{\tau_d^{(1)}(\be)} = \min_{1\le i\le q} \sum_{j=1}^d A_{ij}>0,$$
the final domination being a consequence of each row of $\bA$ having at least one positive entry.

On the other hand, suppose that $\tau_{q,d}^{(q,1)} (\bA)>0$ and  $\bA$ has a trivial row.  Then for any $\bz\in \E_d^{(1)}$, we have
\[\tau_q^{(q)} (\bA\bz) = \min_{1\le i\le q} \sum_{j=1}^d A_{ij}z_j =0.\]
This implies
$$ \tau_{q,d}^{(q,1)} (\bA) = \sup_{\bz \in \E_d^{(i)}} \frac{\tau_q^{(q)} (\bA\bz)}{\tau_d^{(1)}(\bz)}=0,$$ which is a contradiction. Hence $\bA$ cannot have a trivial row.
\item[(e)] The first inequality follows from (b).
Moreover $$\tau_{q,d}^{(1,1)}(\bA) = \sup_{\bz\in \partial \aleph_d^{(1)}} (\bA\bz)^{(1)} = \sup_{\bz\in \E_d^{(1)}: z^{(1)}=1} (\bA\bz)^{(1)} \le d\max_{1\le i \le q, 1\le j \le d} A_{ij} < \infty.$$
\end{enumerate}
\end{proof}
For a deterministic matrix $\bA\in\R_+^{q\times d}$ and $C\subseteq \R_+^q$, the {\em pre-image of $C$} is given by $$\bA^{-1}(C) = \{\bz \in \R_+^d: \bA \bz \in C\}.$$The following lemma characterizes the mapping of the subspaces of $\R_+^{d}$ under the linear map $\bA$ and is key to the results to follow.

\begin{lemma} \label{Lemma 3.4}
Let $\bA \in \R_+^{q\times d}$ be a deterministic matrix with all rows non-trivial.
Then for fixed $i\in\{1,\dots,d\}$ and fixed $k\in\{1,\dots,q\}$, the following are equivalent:
\begin{itemize}
    \item[(a)] $\bA^{-1}(\E_q^{(k)}) \subseteq \E_d^{(i)}$.
    \item[(b)] $0<\tau_{q,d}^{(k,i)} (\bA) <\infty$.
\end{itemize}
\end{lemma}

%
%

\begin{proof}
(a)$\Rightarrow$(b): \, Let $\bA^{-1}(\E_q^{(k)}) \subseteq \E_d^{(i)}$.
First suppose that $\tau_{q,d}^{(k,i)} (\bA)=0$. Hence by definition, from \eqref{mtau} we have that
$\tau_{q}^{(k)} (\bA\bz)=(\bA\bz)^{(k)}=0$\
for every $\bz\in \E_d^{(i)}$.
Thus $$\bA^{-1}(\E_q^{(k)}) \cap \E_d^{(i)}=\emptyset$$
contradicting the premise.\\
Now suppose that $\tau_{q,d}^{(k,i)} (\bA)=\infty$.
Let $M=\tau_q^{(1)}(\bA\be)$ where $\be=(1,1,\ldots,1)^\textsf{T}\in\R_+^d$.
Then there exists a $\bz \in \partial \aleph_d^{(i)}= \{\bz\in \R_+^d: z^{(i)}=1\}$ such that $\tau_q^{(k)}(\bA\bz)\ge M+d$.
Fix such a $\bz$ and without loss of generality assume that $z_1\ge z_2\ge \ldots\ge z_d$ (otherwise we may arrange {columns} of $\bA$ accordingly). Hence $z^{(i)} =z_i=1$. Define $\bz^*\in \R_+^d$ by converting the last $d-i$ components of $\bz$ to 1. Hence
\[\bz^*=(z_1,\ldots,z_{i-1},1,\ldots,1)^{\top}.\]
Since the components of $\bz^*$ and $\bz$ are ordered and component-wise $\bz^*\ge \bz$, we have $\tau_q^{(k)}(\bA\bz^*)\ge \tau_q^{(k)}(\bA\bz)\ge M+d$.
Now, define \[\bz^*_{\be}:=\bz^*-\be= (z_1-1,\ldots,z_{i-1}-1,0,\ldots,0)^{\top}.\]
Clearly $\bz^*_{\be} \in \R_+^{d}$ as well as $\bz^*_{\be}\notin \E_d^{(i)}$ since $z_{\be}^{*(i)}=z_{\be,i}^*=0$. Note that $\tau_q^{(k)}(\bA\bz^*)\ge M+d$ means at least $k$-elements of $\bA\bz^*$ are larger than $M+d$, whereas $\tau_q^{(k)}(\bA\be)\le \tau_q^{(1)}(\bA\be)=M$
by definition.
Hence all elements of $\bA\be$ are at most $M$.
Since $\bA\bz^*_{\be}=\bA\bz^*-\bA\be$, at least $k$ elements of $\bA\bz^*_{\be}$ are greater or equal to $d$.
Therefore, $\tau_q^{(k)}(\bA\bz^*_{\be}) \ge d>0$. Thus $\bA\bz^*_{\be}\in \E_q^{(k)}$ which is a contradiction.

(b)$\Rightarrow$(a):  Let $\bx\in \E_q^{(k)}$. Then $\tau_q^{(k)}(\bx)>0$.
 Furthermore, let $\bA^{-1}(\bx):=\{\bz\in\R_+^d:\bA\bz=\bx\}$
 and let $\bz_{\bx}\in \bA^{-1}(\bx) \subseteq \R_+^d $. Then by Lemma \ref{lem:tauqi}(a),
 \begin{eqnarray*}
    \tau_d^{(i)}(\bz_{\bx})\geq\frac{\tau_q^{(k)}(\bA\bz_{\bx})}{\tau_{q,d}^{(k,i)}(\bA)}=\frac{\tau_q^{(k)}(\bx)}{\tau_{q,d}^{(k,i)}(\bA)}>0,
 \end{eqnarray*}
implying $\bz_{\bx} \in \E_d^{(i)}$. Hence $\bA^{-1}(\E_q^{(k)}) \subseteq \E_d^{(i)}$.
\end{proof}

\bexam\label{Ex1}
The following example illustrates the equivalence shown in Lemma \ref{Lemma 3.4}.
 Suppose that
\begin{align*}
\bA = \begin{bmatrix}
       1 & 1 & 1 & 0         \\[0.3em]
       1 & 1 & 0 & 1         \\[0.3em]
       1 & 0 & 1 & 1         \\[0.3em]
       0 & 1 & 1 & 1         \\[0.3em]
     \end{bmatrix}
     \end{align*}
and $\bz=(z_1,z_2,z_3,z_4)^{\top}$. Then
\beao
\bx = \bA\bz = (z_{1}+z_{2}+z_{3},z_{1}+z_{2}+z_{4},z_{1}+z_{3}+z_{4},z_{2}+z_{3}+z_{4})^{\top}.
\eeao
For $k=q=4$ we find
\beao
\tau_{4,4}^{(4,1)} (\bA) & = & \sup_{\bz\in \E_{4}^{(1)}} \frac{x^{{(4)}}}{z^{{(1)}}} = 3<\infty, \quad \quad \tau_{4,4}^{(4,2)} (\bA) = \sup_{\bz\in \E_{4}^{(2)}} \frac{x^{{(4)}}}{z^{{(2)}}} = 3<\infty,\\
 \tau_{4,4}^{(4,3)} (\bA) & = & \sup_{\bz\in \E_{4}^{(3)}} \frac{x^{{(4)}}}{z^{{(3)}}} = \infty.
\eeao
The supremum value of 3 in the first two cases is attained at $\bz=(z,z,z,z)^{\top}$ for $z>0$. The final equality is attained by using $\bz^*=(z^4,z^3,z^2,z)^{\top}$ for $z>0$, where $\bz^*\in\E_{4}^{(3)}$.
Hence according to Lemma~\ref{Lemma 3.4} we have
\[ \bA^{-1}(\E_{4}^{(4)}) \subseteq \E_{4}^{(2)} \quad (\text{and by inclusion also $\E_{4}^{{(1)}}$}).\]
This means that the pre-image $\bA^{-1}(\E_{4}^{(4)})$ contains vectors $\bz\in\R_+^4$, whose largest two components are positive, and the other two components can be either zero or positive.

This example can be compared to \citet[Lemma 2.2]{janssen:drees:2016} where only  $\tau_{4,4}^{(4,4)} (\bA)$ is considered, which  for this example  is infinite by
Lemma~\ref{lem:tauqi}(c). The only choice for $\bA$ where  $\tau_{4,4}^{(4,4)} (\bA)<\infty$ are permutations of diagonal matrices with positive diagonal entries; see Remark \ref{rem:jandrees} (iii).
\eexam

\subsection{Main Result}\label{subsec:mainBreiman}

 The key result extending Theorems \ref{th31} and \ref{lem:jd2017thm23}, incorporating general random matrices $\bA\in \R_+^{q\times d}$ and a wide variety of tail sets, is provided in this section. If $\bZ \in \MRV(\alpha,\mu)$ with asymptotically independent components, implying $\mu(\R_+^d\setminus\bCA^{(d-1)}_{d})=\mu(\{\bz\in \R_d^+:z^{(d)}>0\}) =0$, we may seek and find multivariate regular variation in subcones $\E^{(i)}_{d}$ for $i=1,\ldots, d$ as seen in Section~\ref{subsec:hrv}. Theorem~\ref{thm:breiman} provides the appropriate non-null limit and its rate in the presence of such  regular variation for  $\bA\bZ$.




For $k=1,\ldots, q$ and  $\omega\in\Omega$ define $\bA_\omega:=\bA(\omega)$ and
 $$ i_k(\bA_\omega)=\arg\max\{j\in\{1,\ldots,d\}: \tau_{q,d}^{(k,j)}(\bA_\omega)<\infty\},$$
which creates a partition of $\Omega$ given by
\[\Omega_i^{(k)}:= \{\omega\in \Omega: i_k(\bA_{\omega})=i\}, \quad i=1,\ldots,d.\]
We write $\P_i^{(k)} (\,\cdot\,):=\P(\,\cdot\,\cap\,\Omega_i^{(k)})$ and $\bE_i^{(k)}[\,\cdot\,]:=\bE[\,\cdot\,\bone_{\Omega_i^{(k)}}]$.
This means, for fixed $k$, we summarize all $\omega\in\Omega$, such that $\bA_\omega$ yields the same $i_k$, and
we work on measure spaces $(\Omega_i^{(k)}, \mathcal{F}\cap \Omega_i^{(k)}, \P_i^{(k)} )$ indexed by $i=1,\ldots,d$.
 \begin{theorem}\label{thm:breiman}
   Let $i\in\{1,\ldots, d\}$ be fixed and  $\bZ\in \R_+^d$ a random vector  such that $\bZ \in \MRV(\alpha_{i},b_{i}, \mu_{i},\E^{(i)}_d)$ with canonical choice of $b_i$  as in Definition~\ref{def:bit}.
Also let $\bA \in \R_+^{q \times d}$ be a  random matrix with almost surely no trivial rows  independent of $\bZ$.
Furthermore, assume that the following conditions are satisfied for some $k\in\{1,\ldots,q\}$:
\begin{enumerate}[(i)]
 \item
 for some $\delta=\delta(i,k)>0$ we have
 $$\bE_i^{(k)}\left[\tau_{q,d}^{(k,i)}(\bA)^{\alpha_{i}+\delta}\right]
:=\int\limits_{\Omega_i^{(k)}}  \tau_{q,d}^{(k,i)}(\bA)^{\alpha_{i}+\delta} \; \mathrm d\P <\infty,$$
 \item
 {$\mu_{i}(\widetilde \bCA_d^{(i)}(j))>0$}  for all $j=1,\ldots,\binom{d}{i}$.
\end{enumerate}
  Then we have
  \begin{align}\label{eq:multBreimannew}
 \frac{\P_i^{(k)}(\bA\bZ \in t \,\cdot\,)}{\P(\tau_{d}^{(i)}(\bZ)>t )} \to \bE_i^{(k)}\left[ \mu_{i}(\{\bz\in\E_d^{(i)}:\bA\bz \in \,\cdot\,\})\right] =: \overline{\mu}_{i,k}(\cdot),\quad\tto,
  \end{align}
  in $\M(\E^{(k)}_{q})$.
    \end{theorem}

\begin{proof}
If $\P(\Omega_i^{(k)})=0$ then \eqref{eq:multBreimannew} is trivially satisfied because the left and right hand side are zero.
Thus we assume that $\P(\Omega_i^{(k)})>0$.
 Let   $C\subset \E_q^{(k)}$ be a Borel set which is bounded away from $\bCA^{(k-1)}_{q}$ and satisfies $\bE^{(k)}_i\left[\mu_i(\partial\bA^{-1}(C))\right] =0$. Then there exists a constant $\delta_C$ such that ${\tau_q^{(k)}(\bx)} = x^{(k)}>\delta_C$
 for all {$\bx\in C$.} Using Lemma~\ref{lem:tauqi}(a), we have for all $t>0$, $M>0$
 \begin{eqnarray*}
    \P_i^{(k)}(\bA\bZ\in tC,\tau_{q,d}^{(k,i)}(\bA)>M)
    &\leq&\P_i^{(k)}(\tau_q^{(k)}(\bA\bZ)>t\delta_C,\tau_{q,d}^{(k,i)}(\bA)>M)\\
        &\leq& \P(\tau_{q,d}^{(k,i)}(\bA){\tau_d^{(i)}}(\bZ)>t\delta_C,\tau_{q,d}^{(k,i)}(\bA)>M,\Omega_i^{(k)}).
 \end{eqnarray*}
 Since $\tau_d^{(i)}(\bZ)=Z^{(i)}\in\RV_{-\alpha_{i}}$,
 and $\bA$ and $\bZ$ are assumed to be independent, the univariate version of Breiman's Theorem in combination with $\bE_i^{(k)}[\tau_{q,d}^{(k,i)}(\bA)^{\alpha_{i}+\delta}] <\infty$
 yields
 \begin{eqnarray*}
    \limsup_{t\to\infty}\frac{\P_i^{(k)}(\bA\bZ\in tC,\tau_{q,d}^{(k,i)}(\bA)>M)}{\P(\tau_{d}^{(i)}(\bZ)>t )}
        &\leq&\limsup_{t\to\infty}\frac{\P(\bone_{\{\tau_{q,d}^{(k,i)}(\bA)>M\}\cap \Omega_i^{(k)}} \tau_{q,d}^{(k,i)}(\bA) \tau_d^{(i)}(\bZ)>t\delta_C)}{\P(\tau_{d}^{(i)}(\bZ)>t )}\\
        &=&\delta_C^{-\alpha_{i}}\bE[\tau_{q,d}^{(k,i)}(\bA)^{\alpha_{i}}\bone_{\{\tau_{q,d}^{(k,i)}(\bA)>M\}\cap\Omega_i^{(k)}}].
 \end{eqnarray*}
 Note that $\bA^{-1}(C):=\{\bz\in\R_+^d:\bA\bz\in C\}$ is again a.s. bounded away from $\bCA^{(i-1)}_{d}$, since for {$\bx\in C$}, $\omega\in {\Omega_i^{(k)}}$,
 and {$\bz_{\bx}\in \bA_{\omega}^{-1}(C)\subseteq \R_+^d$ we have by Lemma~\ref{lem:tauqi}(a),}
 \begin{eqnarray} \label{eq:1}
    \tau_d^{(i)}(\bz_{\bx})\geq\frac{\tau_q^{(k)}(\bA_\omega\bz_{\bx})}{\tau_{q,d}^{(k,i)}(\bA_\omega)}=\frac{\tau_q^{(k)}(\bx)}{\tau_{q,d}^{(k,i)}(\bA_\omega)}
    >\frac{\delta_C}{\tau_{q,d}^{(k,i)}(\bA_\omega)}>0
 \end{eqnarray}
 and, thus, {$\P_i^{(k)}(\bA^{-1}(C)\subseteq \E_d^{(i)})=1$}.
  Hence abbreviating $\ba:=\bA_\omega$ and conditioning on $\bA$, by independence of $\bZ$ and $\bA$, we obtain
 \begin{eqnarray*}
    \lim_{t\to\infty}\frac{\P_i^{(k)}(\bA\bZ\in tC,\tau_{q,d}^{(k,i)}(\bA)\leq M)}{\P(\tau_{d}^{(i)}(\bZ)>t )} &=&
        \lim_{t\to\infty}\int_{\{\tau_{q,d}^{(k,i)}(\ba)\leq M\}}\frac{\P(\bZ\in t\ba^{-1}(C))}{\P (\tau_{d}^{(i)}(\bZ)>t)}\,\mathrm d\P_i^{(k)}(\ba)\\
        &=&\int_{\{\tau_{q,d}^{(k,i)}(\ba)\leq M\}}\mu_{i}(\ba^{-1}(C))\,\mathrm d\P_{i}^{(k)}(\ba)\\
        &=&\bE_i^{(k)}\left[ \mu_{i}\left(\bA^{-1}(C)\bone_{\{\tau_{q,d}^{(k,i)}(\bA)\leq M\}}\right)\right],
 \end{eqnarray*}
 where we used for the third equality that  $\bE_i^{(k)}[\mu_{i}(\partial \bA^{-1}(C))]=0$  in combination with Pratt's lemma \citep{pratt:1960}, since for $\tau_{q,d}^{(k,i)}(\bA_\omega)\leq M$ we have for the integrand
 \begin{eqnarray*}
        \frac{\P(\bZ\in t \bA_\omega^{-1}(C))}{\P(\tau_{d}^{(i)}(\bZ)>t)} & \leq & \frac{\P(\tau_{q,d}^{(k,i)}(\bA_\omega)\tau_d^{(i)}(\bZ)>t\delta_C)}{\P(\tau_{d}^{(i)}(\bZ)>t)}\\
           & \leq & \frac{\P(M\tau_d^{(i)}(\bZ)>t\delta_C)}{\P(\tau_{d}^{(i)}(\bZ)>t)}\to M^{\alpha_{i}}\delta_C^{-\alpha_{i}}, \quad t\to\infty.
 \end{eqnarray*}
We need to show that $\bE_i^{(k)}[\mu_{i}(\bA^{-1}(C))]<\infty$.
Define $B_d^{(i)}(\delta):=\{\bz\in \R_+^d:\tau_d^{(i)}(\bz)\leq \delta\}$.
 By the homogeneity of $\mu_{i}$  and \eqref{eq:1} we have
 \begin{eqnarray*}
   \bE_i^{(k)}\left[\mu_{i}(\bA^{-1}(C))\right]
   &\leq & \bE_i^{(k)}\left[\mu_{i}\left(B_d^{(i)}\left(\delta_C/\tau_{q,d}^{(k,i)}(\bA)\right)^c\right)\right]\\
   & = &\mu_{i}\left(\left(B_d^{(i)}(\delta_C)\right)^c\right)\bE_i^{(k)}\left[\tau_{q,d}^{(k,i)}(\bA)^{\alpha_{i}}\right]<\infty.
 \end{eqnarray*}
To finish the proof it remains to show that $\bE_i^{(k)}[\mu_{i}(\bA^{-1}(\E_q^{(k)}))] >0.$

\textbf{Case 1:} {Suppose $1\le i<d$.  Let $\omega\in\Omega_i^{(k)}$ .}
We know from Lemma~\ref{Lemma 3.4} that
$\bA^{-1}_\omega(\E_q^{(k)}) \subseteq \E_d^{(i)}$.
By definition, $\bCA_d^{(i)}\backslash \bCA_d^{(i-1)} \subset \E_d^{(i)}$.
 We claim that  $$(\bCA_d^{(i)}\backslash \bCA_d^{(i-1)})\cap \bA^{-1}_{\omega}(\E_q^{(k)})\not=\emptyset.$$ If not, then we have $\bA^{-1}_{\omega}(\E_q^{(k)}) \subseteq \E_d^{(i)} \setminus (\bCA_d^{(i)}\backslash \bCA_d^{(i-1)})
 =\R_+^d\backslash  \bCA_d^{(i)}= \E_d^{(i+1)}$.
 Therefore by Lemma \ref{Lemma 3.4}, $\tau_{q,d}^{(k,i+1)}(\bA_{\omega}) <\infty$. But this is a contradiction to the definition of $\Omega_i^{(k)}$ {since $\omega\in\Omega_i^{(k)}$}.\\
 So let $\bz\in (\bCA_d^{(i)}\backslash \bCA_d^{(i-1)})\cap \bA^{-1}_{\omega}(\E_q^{(k)})$.
 Then by \eqref{eq:cminusc}, we have $\bz \in \widetilde \bCA_d^{(i)}(j^*)$ for some $1\le j^*\le \binom{d}{i}$.
 Let $I_{\bz}:=\{j\in\{1,\ldots,d\}:\,z_j>0\}$. {Clearly,
 $$\bCA_d^{(i)}(j^*)=\{\bz\in \R_+^d: z_j>0 \; \text{for } j\in I_{\bz}\mbox{ and }\; z_j=0 \; \text{for } j\in \{1,\ldots,d\}\backslash  I_{\bz}\}.$$
Hence for every $\bz^* \in\bCA_d^{(i)}(j^*)$} we have that some component of $\bA_{\omega}\bz^*$ is positive if and only if the corresponding component of $\bA_{\omega}\bz$ is positive, since $\bA_{\omega}$ has only non-negative entries.
 Thus $\bA_{\omega}\bz^*\in \E_q^{(k)} $, i.e., $\bz^*\in \bA_{\omega}^{-1}(\E_q^{(k)})$.
Hence, we get that
 \[\widetilde\bCA_d^{(i)}(j^*) \subseteq \bA_{\omega}^{-1}(\E_q^{(k)}) \subseteq \E_d^{(i)}.\]
{Since} due to assumption (ii), $\mu_{i}$ has positive mass on each of the $\binom{d}{i}$ hyperplanes $\widetilde \bCA_d^{(i)}(j)$, this results in
 \[\mu_{i}(\bA_{\omega}^{-1}(\E_q^{(k)})) \ge \mu_{i}(\widetilde \bCA_d^{(i)}(j^*))\geq \min_{j} \mu_{i}(\widetilde \bCA_d^{(i)}(j)) >0, \]
and
$$\bE_i^{(k)}[\mu_{i}(\bA^{-1}(\E_q^{(k)}))]\geq \min_{j} \mu_{i}(\widetilde \bCA_d^{(i)}(j)) >0,$$
which proves the claim for $1\leq i<d$.

\textbf{Case 2:} {Suppose $i=d$. Let $\omega\in\Omega_d^{(d)}$}. Take $\bz\in \bA_{\omega}^{-1}(\E_q^{(k)}) \subseteq \E_d^{(d)}$, then all components of $\bz$ and  $\bA_{\omega}\bz\in\E_q^{(k)}$ are positive. Thus $\bA_{\omega}$ has no trivial row and we get that for every $\bz^*\in \E_d^{(d)}$ also $\bA_{\omega}\bz^*$ has only positive components, i.e., $\bA_{\omega}\bz^*\in \E_q^{(k)}$.
 This results in
 $\E_d^{(d)}=\bA_{\omega}^{-1}(\E_q^{(k)})$ and $\bE_d^{(k)}[\mu_{d}(\bA^{-1}(\E_q^{(k)}))]=\mu_{d}(\E_d^{(d)})>0$.
 \end{proof}


\begin{remark} \label{remarkCA}
 The condition that  {$\mu_{i}(\widetilde \bCA_d^{(i)}(j))>0$}  for all $j=1,\ldots,\binom{d}{i}$ could be relaxed to  {$\mu_{i}(\widetilde \bCA_d^{(i)}(j))>0$}  for \emph{at least one} $j\in\{1,\ldots,\binom{d}{i}\}$, but showing that the limit measure is non-zero turns out to be a cumbersome exercise and needs to be done with proper care. In many examples, the measures $\mu_i$ turn out to be exchangeable with respect to their co-ordinates and the assumption being true for all $j=1,\ldots,\binom{d}{i}$ is not uncommon.
 One such example is given in Example~\ref{ex:indpar} where $Z_{i}$ are iid Pareto$(\alpha)$ for $\alpha>0$ and we have $\mu_i(\widetilde\bCA_d^{(i)} (j))>0$ for all $j= 1,\ldots, \binom{d}{i}$.
 \end{remark}

 Theorem~\ref{thm:breiman} provides regular variation limit measures for sets in $\E_q^{(k)}$ restricted to $\Omega_i^{(k)}$, whenever the two conditions are satisfied and, as a consequence, we have the following limit probabilities.

\begin{theorem}\label{prop:mainreg}
Let  $\bZ\in \R_+^d$ be a random vector  such that for all $i =1,\ldots,d$, we have
$\bZ \in \MRV(\alpha_{i},b_{i}, \mu_{i},\E^{(i)}_d)$ with canonical choice of $b_i$ as in Definition~\ref{def:bit}. Moreover,
for all {$i=1,\ldots,d-1$} we assume $b_i(t)/b_{i+1}(t)\to \infty$ as $t\to\infty$.
Let $\bA \in \R_+^{q \times d}$ be a  random matrix  with almost surely no trivial rows independent of $\bZ$.
Furthermore, let $C\subset \E_q^{(k)}$ for $k\in\{1,\ldots,q\}$
be a  Borel set bounded away from $\bCA_q^{(k-1)}$ with $\bE_i^{(k)}[\mu_i(\partial\bA^{-1}(C))]=0$ for all $i=1,\ldots,d$.\\
Suppose further that
\begin{enumerate}[(i)]
\item
 $\bE_i^{(k)}\left[\tau_{q,d}^{(k,i)}(\bA)^{\alpha_{i}+\delta}\right] <\infty$ for some $\delta=\delta(i,k)>0$,
 \item   {$\mu_{i}(\widetilde \bCA_d^{(i)}(j))>0$}  for all $j=1,\ldots,\binom{d}{i}$.
\end{enumerate}
 Then the following results hold.
\begin{enumerate}[(a)]
\item We have
\begin{eqnarray}\label{eq:hidsum}
\P(\bA\bZ \in t C)
    & = &\sum\limits_{i=1}^d
{\P(Z^{(i)}>t)} \, \left[\bE_i^{(k)}[\mu_{i}(\bA^{-1}(C))] +o(1)\right], \quad t\to\infty.
\end{eqnarray}
\item Define
\begin{align}\label{eq:ik*}
i_k^* := \arg\min\{i\in \{1,\ldots,d\}: \P(\Omega_i^{(k)})>0\}.
    \end{align}
    Then we have
 \begin{align*}
 \frac{\P(\bA\bZ \in tC)}{\P(Z^{(i_k^*)}>t )} \to \bE_{i_k^*}^{(k)}\big[ \mu_{i_k^*}(\bA^{-1}(C))\big] \, =\overline{\mu}_{i_{k}^{*},k}(C)=:
 \overline{\mu}_{k}(C),\quad\tto,
  \end{align*}
   in $\M(\E^{(k)}_{q})$. Hence, $\bA\bZ \in \MRV(\alpha_{i_{k}^{*}},\overline{\mu}_{k},\E_q^{(k)})$.

\end{enumerate}

\end{theorem}

\begin{proof}~
\begin{enumerate}[(a)]
    \item   Since $\{\Omega_i^{(k)}, 1\le i \le d\}$ forms a partition of $\Omega$,
$\P(\bA\bZ \in t C) =\sum_{i=1}^d \P_i^{(k)}(\bA\bZ \in t C)$.
Hence using \eqref{eq:multBreimannew} and observing that
$$ \P(\tau_{d}^{(i)}(\bZ)>t )=\P(Z^{(i)}>t) \sim 1/b_{i}^{\leftarrow}(t),\quad t\to\infty,$$
we have
\begin{align*}
\P(\bA\bZ \in t C) = \sum\limits_{i=1}^d
\left[{\P(Z^{(i)}>t)} \, \big[\bE_i^{(k)}[\mu_{i}(\bA^{-1}(C))] + o(1)\big]\right], \quad t\to\infty.
\end{align*}

\item Now, since $\{\Omega_i^{(k)}, 1\le i \le d\}$ forms a partition of $\Omega$, there exists a $j\in\{1,\ldots,d\}$  with $\P(\Omega_{j}^{(k)})>0$ and hence, $i_k^*$ is well-defined.
Note that using \eqref{eq:hidsum}, we have for any Borel set $C\subset \E_q^{(k)}$  bounded away from $\bCA_q^{(k-1)}$ with $\bE_i^{(k)}(\mu_i(\partial \bA^{-1}(C)))=0$ for $i=1,\ldots,d$, the following asymptotic behavior:
\begin{eqnarray*}
\frac{\P(\bA\bZ \in tC)}{\P(Z^{(i_k^*)}>t)}
& = & \Big[\bE_{i_k^*}^{(k)}[\mu_{i_k^*}(\bA^{-1}(C))]+ o(1)\Big]\\
&&+ \sum_{i=i_k^*+1}^d \left[\frac{{\P(Z^{(i)}>t)}}{\P(Z^{(i_k^*)}>t)} \, \bE_i^{(k)}[\mu_{i}(\bA^{-1}(C))]+\frac{o(\P(Z^{(i)}>t)))}{\P(Z^{(i_k^*)}>t)}\right]\\
& \to& \bE_{i_k^*}^{(k)}\big[\mu_{i_k^*}(\bA^{-1}(C))\big] = \overline{\mu}_k(C), \quad  t\to\infty,
\end{eqnarray*}
    since for all $i=i_k^*+1,\ldots, d$ we have {$b_{i_k^*}(t)/b_{i}(t) \to \infty$}, and hence
    \begin{align*}
        \frac{{\P(Z^{(i)}>t)}}{\P(Z^{(i_k^*)}>t)} \sim \frac{b_{i_k^*}^{\leftarrow}(t)}{b_i^{\leftarrow}(t)} \to 0,\quad t\to\infty.
    \end{align*}
    Hence $\bA\bZ \in \MRV(\alpha_{i_{k}^{*}},\overline{\mu}_{k},\E_q^{(k)})$.
\end{enumerate}
\end{proof}

\begin{remark}\label{rem:aboutCA}
When we assume  $\bZ \in \MRV(\alpha_{i},b_{i}, \mu_{i},\E^{(i)}_d)$ for all $i=1,\ldots, d$, with \linebreak $b_i(t)/b_{i+1}(t)\to \infty$ as $t\to\infty$ for 
all {$i=1,\ldots,d-1$}, it results in restricting the supports for the measures $\mu_i$ to  $\bCA_d^{(i)}\setminus \CA_d^{(i-1)}$; a specific case is discussed in Example \ref{ex:indpar}.
\end{remark}


\begin{remark}
 If $\bZ$ has asymptotically independent components, and each component has distribution tail $\P(Z_j>t)\sim \kappa_j t^{-\al}$ as $t\to\infty$ for some $\al,\kappa_j>0$, then we get a generalization of Theorem~3.2 of \citet{kley:kluppelberg:reinert:2016}. We investigate such structures further in the next section.
\end{remark}

 The following example illustrates the image and pre-image of sets under the map $\bA:\bz\mapsto \bA\bz=\bx$ as well as the regions, where the limit measure is positive in a 3-dimensional setting.
We emphasize that for this example our theory is not really necessary, the calculations can be done by hand, but it  helps in clarifying the ideas and the notation needed for more complex examples to follow.

\bexam \label{Ex2}
Let $\bZ=(Z_{1},Z_{2},Z_{3})^{\top}$ have iid Pareto($\alpha$) marginal distributions with $\P(Z_i>z)=z^{-\alpha}, z>1$ for some $\alpha>0$ as in Example~\ref{ex:indpar}.
Then $\bZ\in \MRV(i\alpha,b_i,\mu_i,\E_3^{(i)})$ for $i=1,2,3$ where $b_1(t)= (3t)^{1/\alpha}, b_2(t)= (3t)^{1/(2\alpha)}$ and $b_3(t)= t^{1/(3\alpha)}$ are the canonical choices and
\begin{align*}
   & \mu_1\Big(\bigcup\limits_{i=1}^3 \left\{\bv\in \R_+^3: v_i>z_i\right\}\Big) = \frac 13(z_1^{-\alpha}+z_2^{-\alpha}+z_3^{-\alpha}),\\
    & \mu_2\Big(\bigcup\limits_{1\le i\neq j\le 3}\left\{\bv\in \R_+^3: v_i>z_i, v_j>z_j\right\}\Big) =  \frac 13\{(z_1z_2)^{-\alpha}+(z_2z_3)^{-\alpha}+(z_3z_1)^{-\alpha}\},\\     & \mu_3\left((z_1,\infty)\times(z_2,\infty)\times(z_3,\infty)\right) =  (z_1z_2z_3)^{-\alpha}
\end{align*}
for $z_1,z_2,z_3>0$. Consider the matrix
\begin{align*}
\bA = \begin{bmatrix}
       1 & 1 & 0         \\[0.3em]
       0 & 1 & 1         \\[0.3em]
       1 & 0 & 1
     \end{bmatrix}.
     \end{align*}
Then, under the map $\bA:\bz\mapsto \bA\bz=\bx$, the  region $C=(1,\infty)^{3}\subset\E_{3}^{(3)}$  has pre-image given by 
$$\bA^{-1}(C)= \{\bz\in \R_{+}^{3}: z_{1}>1,z_{2}>1\}\cup \{\bz\in \R_{+}^{3}: z_{2}>1,z_{3}>1\}\cup \{\bz\in \R_{+}^{3}: z_{3}>1,z_{1}>1\}.$$
It is easy to check that $i_3^*=2$, as defined in \eqref{eq:ik*}. Hence, with $\mu_2({\bA}^{-1}(C))=1$ we obtain
\[\P(\bA\bZ\in tC)  \sim \P(Z^{(2)}>t) \mu_2({\bA}^{-1}(C)) \sim 3t^{-2\alpha},\quad t\to\infty.\]
\begin{figure}[t]
    \centering
    \includegraphics[width=0.4\textwidth]{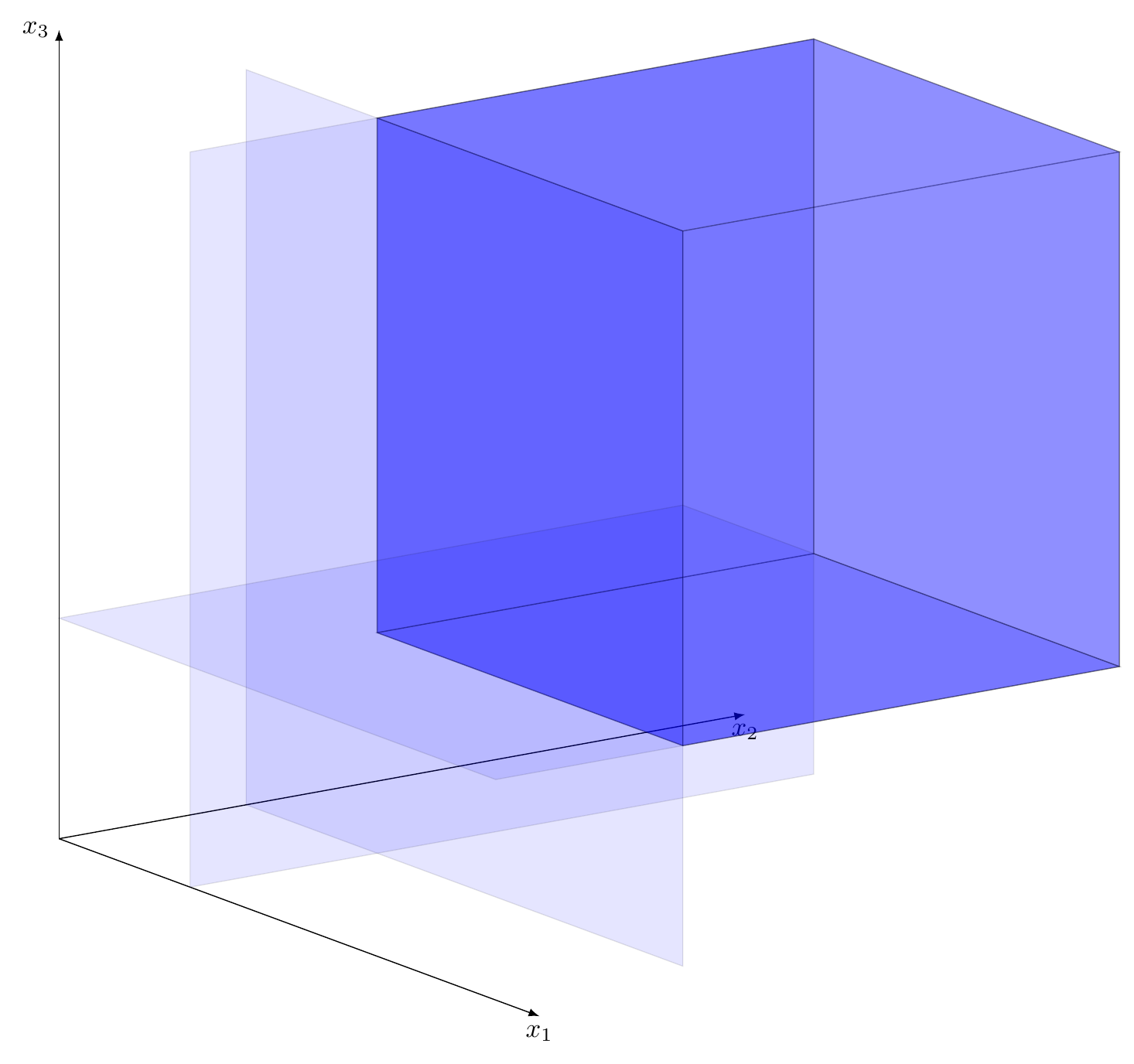} \hspace{1cm}
    \includegraphics[width=0.4\textwidth]{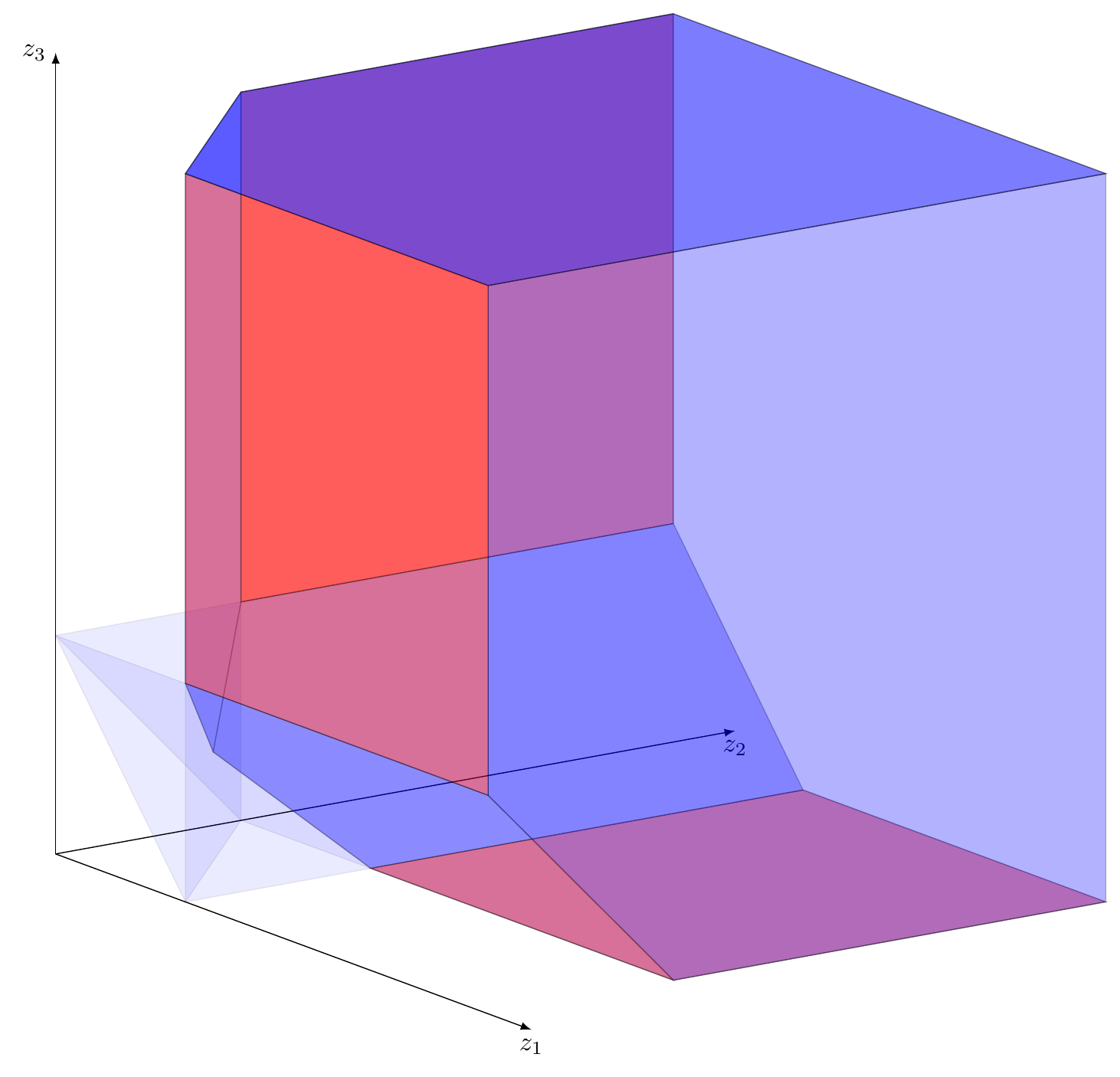}
    \caption{The left plot has the region $C=(\bone,\infty)^3$ in blue in $\bx=(x_1,x_2,x_3)$ co-ordinates. The right plot has the region $\bA^{-1}(C)$ in blue from Example~\ref{Ex2} in $\bz=(z_1,z_2,z_3)$ co-ordinates. The red region shows the support of the measure $\mu_2$.}
    \label{fig:transform}
\end{figure}
Figure \ref{fig:transform} gives a plot of the region $C$ and the transformed region $\bA^{-1}(C)$ colored in blue.
The red region on the right plot shows the support of the measure $\mu_2$.
\eexam

\begin{remark}
In Theorem 3.7, we ascertain  the asymptotic behavior of $\P(\bA\bZ\in tC)$  for certain sets $C\subset\E_q^{(k)}$. Specifically, from Theorem \ref{prop:mainreg} (b) we have
$$\P(\bA\bZ\in tC)\;  = \; {\P(Z^{(i_k^*)}>t)} \,\bE_{i_k^*}^{(k)}[\mu_{i_k^*}(\bA^{-1}(C))]  +o(\P(Z^{(i_k^*)}>t),\quad t\to\infty, $$ where $i_k^*$ is defined in \eqref{eq:ik*}.
If $\bE_{i_k^*}^{(k)}[\mu_{i_k^*}(\bA^{-1}(C))]=0$, we only get $$\P(\bA\bZ\in tC) \, = \,  o(\P(Z^{(i_k^*)}>t), \quad\quad t\to\infty.$$
However, under certain assumptions on $\bA$ and $C$ we can say more about the precise rates, illustrated in the following results.
\end{remark}

\begin{proposition}\label{prop:mainregext}
Let the assumptions and notations of  Theorem \ref{prop:mainreg} hold.
Define
\begin{align}\label{def:iota}
\bar{\iota}= \bar{\iota}_C:= \min\{d, \inf\{i\in\{i_k^*,\ldots,d\}: \bE_i^{(k)}\left[\mu_i(\bA^{-1}(C))\right]>0\}\}.
\end{align}
Suppose for all $i=i^*_k,\ldots,\bar{\iota}-1$ and $\omega\in \Omega_i^{(k)}$ that $\bA_{\omega}^{-1}(C)=\emptyset$.  Then
\begin{align}\label{eq:limiota}
\P(\bA\bZ\in tC)= {\P(Z^{(\bar{\iota})}>t)} \,\bE_{\bar{\iota}}^{(k)}[\mu_{\bar{\iota}}(\bA^{-1}(C))] + o(\P(Z^{(\bar{\iota})}>t)),\quad t\to\infty.
\end{align}
 \end{proposition}

\begin{proof}
Since by assumption, for all $i= i_k^*,\ldots,\bar{\iota}-1$ and $\omega\in \Omega_i^{(k)}$, we have $\bA_{\omega}^{-1}(C)=\emptyset$, we obtain
\begin{align}\label{eq:allizero}
\P_i^{(k)}\left(\bA\bZ\in t C\right) =0.
\end{align}
Therefore due to the definition of $i_k^*$, 
\begin{align*}
 \P(\bA\bZ\in tC)
                          & = \sum_{i=i^*_k}^{d}\P_i^{(k)}\left(\bA\bZ\in tC\right)\\ 
                          & = \sum_{i=\bar{\iota}}^{d}\P_i^{(k)}\left(\bA\bZ\in t C\right)\\ 
                          & = {\P(Z^{(\bar{\iota})}>t)} \,\bE_{\bar{\iota}}^{(k)}[\mu_{\bar{\iota}}(\bA^{-1}(C))] + o(\P(Z^{(\bar{\iota})}>t)),\quad t\to\infty,
                          \end{align*}
using Theorem \ref{thm:breiman}.
\end{proof}

The additional assumption made in Proposition~\ref{prop:mainregext} is often satisfied by random matrix structures.
One such example is a random matrix with only one positive entry in each row.  
Such matrices are for instance proposed in the examples of Section~\ref{ss42}. 
Moreover, if $\bA\in \R_+^{d\times d}$ 
and we follow the assumptions of  \cite[Theorem 2.3]{janssen:drees:2016}, we also obtain such matrices; cf. Remark~\ref{rem:jandrees}(iii). 
The following proposition formalizes the result in this case.

\begin{proposition}\label{prop:mainregextcor}
Let the assumptions and notations of  Theorem \ref{prop:mainreg} hold.
 Moreover, let $C\subset\E_q^{(k)}$ be such that $\displaystyle{C=\bigcup_{l=1}^N \Gamma_l}$ for some $N\in \mathbb{N}$, where each $\Gamma_l$ is of the form:
\begin{align*}
\Gamma&=\{\bx\in \R_+^q: x_{j_1}>\gamma_1,\ldots, x_{j_k}>\gamma_k\}.
\end{align*}
  and $\bE_i^{(k)}[\mu_i(\partial\bA^{-1}(C))]=0$ for all $i=1,\ldots,d$.
  Let $\bar{\iota}$ be defined as in \eqref{def:iota}. If the random matrix $\bA$ has a discrete distribution and has exactly one positive entry in each row, then  \eqref{eq:limiota} holds.
\end{proposition}
\begin{proof} If we show that for all $i=  i^*_k,\ldots,\bar{\iota}-1$ and $\omega\in \Omega_i^{(k)}$, we have $\bA_{\omega}^{-1}(C)=\emptyset$, then applying Proposition \ref{prop:mainregext} we get the result.
 Fix $i\in\{i_k^*,\ldots,\bar{\iota}-1\}$.  By definition, we have $ \bE_i^{(k)}[\mu_i(\bA^{-1}(C))]=0$.
 Suppose there exists $\omega\in \Omega_i^{(k)}$ with $$\bA_{\omega}^{-1}(C)\cap (\CA_d^{(i)}\setminus\CA_d^{(i-1)})\neq\emptyset.$$ Then there exists $\bx^*\in C\subset \E_q^{(k)}$ and $\bz^*\in \CA_d^{(i)}\setminus\CA_d^{(i-1)} \subset \E_d^{(i)}$ with $\bA_{\omega}\bz^*=\bx^*\in C$. \\
Since $\bz^*\in \CA_{d}^{(i)}\setminus\CA_d^{(i-1)}$, exactly $i$ components of $\bz^*$ are positive. Without loss of generality let $z_1^*,\ldots,z_i^*>0.$
Now for any $\bz=(z_1,\ldots,z_i,0,\ldots,0)$ with $z_j\ge z_j^*$, we have $\bA\bz\ge\bA\bz^*=\bx^*$ and by the structure of $C$, we have $\bA\bz\in C.$ Hence
\begin{eqnarray*}
    \{\bz\in \R_+^d: z_j \ge z_j^*,\, j=1,\ldots,i \text{ and } z_{i+1}=\ldots=z_d=0\}
      \subseteq \bA_{\omega}^{-1}(C)\cap  (\CA_{d}^{(i)}\setminus\CA_d^{(i-1)}).
\end{eqnarray*}
Now from assumption (ii) of Theorem \ref{prop:mainreg} and the homogeneity of the measure $\mu_i$, we have
\begin{eqnarray*}
0&<&\mu_i\left( \{\bz\in \R_+^d: z_j \ge z_j^*,\, j=1,\ldots,i \text{ and } z_{i+1}=\ldots=z_d=0\} \right)\\
    &\le& \mu_i(\bA_{\omega}^{-1}(C)) \le \frac{\bE_i^{(k)}[\mu_i(\bA^{-1}(C))]}{\P(\bA=\bA_{\omega})},
\end{eqnarray*}
since $\bA$ has a discrete distribution and $\P(\bA=\bA_{\omega})>0$.
Hence $\bE_i^{(k)}[\mu_i(\bA^{-1}(C))]>0$, which is a contradiction. Thus
  \begin{align}\label{eq:cacaempty}
  \bA_{\omega}^{-1}(C)\cap (\CA_d^{(i)}\setminus\CA_d^{(i-1)})=\emptyset.
  \end{align}
 Now suppose that $\bA_{\omega}^{-1}(C)\neq\emptyset$. Then there exist $\bx \in C\subset \E_q^{(k)}$ and $\bz \in\E_d^{(i)}$ with $\bA_{\omega}\bz=\bx\in C$. Since $\omega\in \Omega_i^{(k)}$, exactly $i$ columns of $\bA_{\omega}$ have at least one positive entry. 
  W.l.o.g. assume these are the first $i$ columns of $\bA_{\omega}$. 
  Then {for}
 \[{\bz}^*:= (z_1,\ldots,z_i,0,\ldots,0) \in \CA_d^{(i)}\setminus\CA_d^{(i-1)}\] 
we have
\[\bA_{\omega}{\bz}^*=\bA_{\omega}\bz=\bx\in C\]
 since columns $i+1,\ldots,d$ of $\bA_{\omega}$ have all entries zero and hence the last $d-i$ entries of $\bz$ or ${\bz}^*$ do not count towards the computation of $\bx$. 
 Hence ${\bz}^*\in \bA_{\omega}^{-1}(C)\cap (\CA_d^{(i)}\setminus\CA_d^{(i-1)})$, which is a contradiction to \eqref{eq:cacaempty}. This gives the statement.
\end{proof}

\section{Bipartite networks}\label{sec:applications}

Risk-sharing in complex systems is often modeled using a graphical network model, one such example being the bipartite network structure for modeling losses in insurance markets or financial investment risk as proposed in  \citet{kley:kluppelberg:reinert:2016,kley:kluppelberg:reinert:2017}.
In these papers, only first order asymptotics of risk measures based on the agents' and market's tail risks are derived.
In the same spirit, but going beyond first order approximations, we consider a vertex set of agents $\mathcal{A}=\{1,\ldots,q\}$  and a vertex set of objects (insurance claims or investment risks) $\mathcal{O}=\{1,\ldots, d\}$.

\begin{figure}[h]
\[\begin{tikzpicture}
    \node at (-1.5,1.5) {\sc{Agents:}};
    \node at (-1.5,-1.3) {\sc{Objects:}};
	\vertex[fill] (A1) at (1,1) [label=above:$A_1$] {};
	\vertex[fill] (A2) at (3,1) [label=above:$A_2$] {};
	\vertex[fill] (A3) at (5,1) [label=above:$A_3$] {};
	\vertex (O1) at (0,-1) [label=below:$O_1$] {};
	\vertex (O2) at (2,-1) [label=below:$O_2$] {};
	\vertex (O3) at (4,-1) [label=below:$O_3$] {};
	\vertex (O4) at (6,-1) [label=below:$O_4$] {};
	\path
		(A1) edge (O1)
		(A1) edge (O4)
		(A2) edge (O1)
		(A2) edge (O3)
		(A3) edge (O2)
		(A3) edge (O3)
		(A3) edge (O4)
	;
\end{tikzpicture}\]
\caption{A bipartite network with $q=3$ agents and $d=4$ objects.}\label{fig:bipnet}
\end{figure}
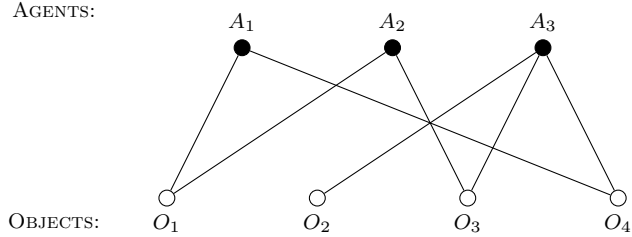

Each agent $k \in \mathcal{A} $ chooses a number of objects $i \in \mathcal{O}$
to connect with. Figure \ref{fig:bipnet} provides an example of such a network. This choice can be random according to some probability distribution.
 A basic model assumes $k$ and $i$ connect with probability
 $$\P(k\sim i) = p_{k i} \in [0,1], \quad k=1,\ldots,q,\, i=1,\ldots,d.$$
Let $Z_{i}$ denote the risk
attributed to the $i$-th object and $\bZ=(Z_{1},\ldots, Z_{d})^{\top}$ forms the risk vector. Assume that the graph creation process is independent of $\bZ$.
The proportion of loss of object $i$ affecting agent $k$ is denoted by
\begin{align*} 
 f_{k}(Z_{i}) = \bone{(k\sim i)} W_{k i} Z_{i},
\end{align*}
where $W_{k i}>0$ denotes the effect of the $i$-th object on the $k$-th agent.
Now define the $q\times d$ {\em weighted adjacency matrix} $\bA=(A_{ki})_{k=1,\ldots,q,i=1,\ldots,d}\in \R_+^{q\times d}$
 by
\begin{align} \label{Aarbitrary}
    A_{k i} &= \bone{(k \sim i)}W_{k i},.
\end{align}

The total exposure of the agents given by $\bX=(X_{1}, \ldots, X_{q})^{\top}$, where $X_{k } = \sum_{i=1}^{d} f_{k }(Z_{i})$ can be represented as
\begin{align*}
    \bX=\bA\bZ.
\end{align*}
Our goal is to find the probability of tail risks of some or all agents in terms of $\bX$.

In \cite{kley:kluppelberg:reinert:2016,kley:kluppelberg:reinert:2017},  proportional weights are used to distribute the insurance loss of object $i$ affecting agent $k$ or to diversify the investment risk of an agent. 
The weights complicate calculations and only affect the values of the limit measures, resulting in different constants, whereas
the rate of convergence remains the same.
As they rather blur the mathematical insight, we work in Section~\ref{ss42}  with unweighted adjacency matrices. 
However, they can be incorporated in the calculations by appropriate multiplications using the independence of $\bA$ and $\bZ$. 
We consider a weighted adjacency matrix in Section~\ref{ss43}, when we investigate dependent objects in contrast to independent ones; see Examples~\ref{ex:det} and \ref{ex:nonind} below.
Throughout we formulate our examples and results in terms of investment risk.

\subsection{Independent objects}\label{ss42}

For illustration we start with an example, which shows how regular variation of $\bX=\bA\bZ$ for independent Pareto-tailed components of $\bZ$ and random adjacency matrices transforms into regular variation of $\bX$. The choice of  $C$ specifies the tail risk, and in this example we calculate the asymptotic tail risk explicitly for two different kinds of sets leading to two different asymptotic rates.

\begin{example}\label{ex:taylornew}
Suppose there are two  products  $P_1$ and  $P_2$ in the market with associated risks $Z_1$ and $Z_2$ respectively, which are independent and  $\P(Z_{i}>z) \sim \kappa_{i} z^{-\al}$ for $\alpha>0$ as $z\to \infty$ with constants $\kappa_{i}>0$ for $i=1,2$. Assume there are three investors, each of whom may either invest in one unit of $P_1$ or one unit of $P_2$ or one unit of both (we assume that they always invest). Hence there are $3\times 3\times3=27$ possible market investments which may be represented by the matrix $\bA$ so that the joint risk 
of the investors is given by $\bX=\bA\bZ.$ Now the 27 possibilities for the matrix $\bA$ are given by
\begin{equation*} \small
   \bA_1= \begin{bmatrix}
    1 & 1\\
    1 & 1\\
    1 & 1
    \end{bmatrix}
    , \; \bA_2=
    \begin{bmatrix}
     1 & 0\\
    1 & 0\\
    1 & 0
     \end{bmatrix}
     , \; \bA_3 =
     \begin{bmatrix}
     0 & 1\\
     0 & 1\\
     0 & 1
     \end{bmatrix}, \quad\quad\quad\quad\quad\quad\quad\quad\quad\quad\quad\quad\quad\quad\quad\quad\quad\quad\quad\quad\quad\quad\;
  \end{equation*}
  \begin{equation*} \small
 \bA_4 =
        \begin{bmatrix}
     1 & 1\\
    1& 0\\
    1 & 0
    \end{bmatrix},\;
    \bA_5 =
        \begin{bmatrix}
     1 & 0\\
    1 & 1\\
    1 & 0
     \end{bmatrix},
     \; \bA_6 =
     \begin{bmatrix}
     1 & 0\\
    1 & 0\\
    1 & 1
    \end{bmatrix},
         \; \bA_7 =
        \begin{bmatrix}
     1 & 1\\
    0& 1\\
    0 & 1
    \end{bmatrix}, \;
    \bA_8 =
        \begin{bmatrix}
     0 & 1\\
    1 & 1\\
    0 & 1
     \end{bmatrix}, \; \bA_9 =
     \begin{bmatrix}
     0 & 1\\
    0 & 1\\
    1 & 1
    \end{bmatrix}, \quad\;\;\;
\end{equation*}
  \begin{equation*} \small
 \bA_{10} =
        \begin{bmatrix}
     1 & 1\\
    1& 1\\
    1 & 0
    \end{bmatrix},\;
    \bA_{11} =
        \begin{bmatrix}
     1 & 0\\
    1 & 1\\
    1 & 1
     \end{bmatrix},
     \; \bA_{12} =
     \begin{bmatrix}
     1 & 1\\
    1 & 0\\
    1 & 1
    \end{bmatrix},
         \; \bA_{13} =
        \begin{bmatrix}
     1 & 1\\
    1& 1\\
    0 & 1
    \end{bmatrix}, \;
    \bA_{14} =
        \begin{bmatrix}
     0 & 1\\
    1 & 1\\
    1 & 1
     \end{bmatrix}, \;
     \bA_{15} =
     \begin{bmatrix}
     1 & 1\\
    0 & 1\\
    1 & 1
    \end{bmatrix},
\end{equation*}

  \begin{equation*} \small
 \bA_{16} =
        \begin{bmatrix}
     1 & 1\\
    1 & 0\\
     0 & 1
    \end{bmatrix},\;
    \bA_{17} =
        \begin{bmatrix}
     1 & 0\\
    1 & 1\\
    0 & 1
     \end{bmatrix},
     \; \bA_{18} =
     \begin{bmatrix}
     1 & 0\\
    0 & 1\\
    1 & 1
    \end{bmatrix},
         \; \bA_{19} =
        \begin{bmatrix}
     1 & 1\\
    0& 1\\
    1 & 0
    \end{bmatrix}, \;
    \bA_{20} =
        \begin{bmatrix}
     0 & 1\\
    1 & 1\\
    1 & 0
     \end{bmatrix}, \; \bA_{21} =
     \begin{bmatrix}
     0 & 1\\
    1& 0\\
    1 & 1
    \end{bmatrix},
\end{equation*}

  \begin{equation*} \small
 \bA_{22} =
        \begin{bmatrix}
     1 & 0\\
    1& 0\\
    0 & 1
    \end{bmatrix},\;
    \bA_{23} =
        \begin{bmatrix}
     0 & 1\\
    1 & 0\\
    1 & 0
     \end{bmatrix},
     \; \bA_{24} =
     \begin{bmatrix}
     1 & 0\\
    0 & 1\\
    1 & 0
    \end{bmatrix},
         \; \bA_{25} =
        \begin{bmatrix}
     0 & 1\\
    0& 1\\
    1 & 0
    \end{bmatrix}, \;
    \bA_{26} =
        \begin{bmatrix}
     1 & 0\\
    0 & 1\\
    0 & 1
     \end{bmatrix}, \; \bA_{27} =
     \begin{bmatrix}
     0 & 1\\
    1 & 0\\
    0 & 1
    \end{bmatrix},
\end{equation*}
and let $q_m:=\P(\bA=\bA_m)\ge 0$ for $m=1,\ldots,27$ such that $\sum_{k=1}^{27}q_m=1$,  $\sum_{m=1}^{15}q_m>0$ and
$q_{16}+q_{19}>0$.\\
Suppose  we want to assess the risk of all investors being above a high threshold $t>0$. Moreover, we also want to find the probability that when all 
risks of the investors are above $t$, the risk of the first investor is larger than that of the second which is larger than the one of the third, i.e., $X_1>X_2>X_3>t$.
Hence given $t>0$ and
\begin{align}
    C_1&=\{\bx\in\E_3^{(3)}: x_i>1, i=1,2,3\} = (\bx,\binfty), \label{set:C1}\\
   C_2&=  \{\bx\in\E_3^{(3)}: x_1>x_2>x_3>1\},  \label{set:C2}
\end{align}
 we want to compute $\P(\bX\in t C_i)$ for $i=1,2.$ First note that $\bZ\in\MRV(\alpha,b_1(t),\mu_1,\E_2^{(1)})$ and $\bZ\in\MRV(2\alpha,b_2(t),\mu_2,\E_2^{(2)})$ where $b_1(t)=((\kappa_1+\kappa_2)t)^{1/\alpha}, b_2(t)=(\kappa_1\kappa_2t)^{1/(2\alpha)}$ and for $(z_1,z_2)\in(0,\infty)^2,$
\begin{align*}
   &\mu_1(\{\bv\in\E_2^{(1)}:v_1>z_1 \text{ or } v_2>z_2\})=\frac{\kappa_1}{\kappa_1+\kappa_2}z_1^{-\alpha}+\frac{\kappa_2}{\kappa_1+\kappa_2}z_2^{-\alpha},\\
  &\mu_2(\{\bv\in\E_2^{(2)}:v_1>z_1,v_2>z_2\})=(z_1z_2)^{-\alpha}.
\end{align*}
In order to compute the necessary probabilities, we first need to compute $i_k^{*}$ as defined in Theorem~\ref{prop:mainreg} based on $\tau_{3,2}^{(k,i)}(\bA_m)$ for $k=1,2,3$ and $m=1,\ldots, 27$. We can check that for $m=1,\ldots, 15$,
\begin{align*}
&\tau_{3,2}^{(k,1)}(\bA_{m}) < \infty, \; k=1,2,3,\\
&\tau_{3,2}^{(k,2)}(\bA_{m}) =\infty, \; k=1,2,3.
\end{align*}
Hence for $m=1,\ldots, 15$, we get $i_k(\bA_m)=1$ for $k=1,2,3.$ On the other hand for $m=16,\ldots, 27$, we observe that
\begin{align*}
&\tau_{3,2}^{(k,1)}(\bA_{m}) < \infty, \; k=1,2,3,\\
&\tau_{3,2}^{(3,2)}(\bA_{m}) < \infty, \quad \tau_{3,2}^{(2,2)}(\bA_{m}) =\tau_{3,2}^{(1,2)}(\bA_{m}) = \infty.
\end{align*}
Therefore for $m=16,\ldots, 27$, we get $i_3(\bA_m)=2, i_2(\bA_m)=1, $ and  $i_1(\bA_m)=1.$\\
Clearly from our assumptions, \begin{align*}
\P(\Omega_1^{(3)})=\sum_{m=1}^{15} q_m=:q_1+Q_1+Q_2>0, \quad  \text{and,} \quad  \P(\Omega_2^{(3)})=\sum_{m=16}^{27} q_m >0,
\end{align*}
where
\begin{align*}
Q_1&=q_2+q_4+q_5+q_6+q_{10}+q_{11}+q_{12}, &&Q_2= q_3+q_7+q_8+q_9+q_{13}+q_{14}+q_{15}.
\end{align*}
Note that both $C_1, C_2 \subset \E_3^{(3)}$. Applying Theorem~\ref{prop:mainreg}(b) results in $i_3^*=1$ and we have
\begin{align*}
\bX \in\MRV(\alpha,b_1(t)=((\kappa_1+\kappa_2)t)^{1/\alpha},\overline{\mu}_3,\E_3^{(3)}),\; \;\quad{where }\;\; \overline{\mu}_3(\cdot) = \bE_1^{(3)}[\mu_1(\bA^{-1}(\cdot))].
\end{align*}
First consider the set $C_1$ as defined in \eqref{set:C1}. Since
\begin{align}
\bE_1^{(3)}[\mu_1(\bA^{-1}(C_1))]&= \sum_{m=1}^{15} q_{m}\mu_1(\{\bz\in\E_2^{(1)} : \bA_m\bz\in C_1\}) \nonumber\\
&=\frac{[(q_1+Q_1)\kappa_1+(q_1+Q_2)\kappa_2]}{\kappa_1+\kappa_2}>0,\label{eq:C1pos}
\end{align}
we have
\beao
\P(\bX \in tC_1) &\sim&  \P(\bZ^{(1)}>t) \bE_1^{(3)}[\mu_1(\bA^{-1}(C_1))]\\
 &\sim&  \left[(q_1+Q_1)\kappa_1+(q_1+Q_2)\kappa_2\right]t^{-\alpha},  \quad t\to\infty.
\eeao
The same result can be obtained from \Cref{prop:mainregext}, since \eqref{eq:C1pos} indicates $\bar{\iota}=\bar{\iota}_{C_1}=1.$
%
Now consider the set $C_2$ as in \eqref{set:C2}. Note that in this case,
\begin{eqnarray*}
    \bE_1^{(3)}[\mu_1(\bA^{-1}(C_2))] &=& \sum_{m=1}^{15}q_m\mu_1(\bA_m^{-1}(C_2))=0, \quad \text{ and,}\\
\bE_2^{(3)}[\mu_2(\bA^{-1}(C_2))]
 &=&q_{16}\mu_2(\{\bz\in\E_2^{(2)} : \bA_{16}\bz\in C_2\})+q_{19}\mu_2(\{\bz\in\E_2^{(2)} : \bA_{19}\bz\in C_2\})\\
& =&\frac12 (q_{16}+q_{19})>0.
\end{eqnarray*}
Hence using notation from  \Cref{prop:mainregext}, we have $\bar{\iota}=\bar{\iota}_{C_2}=2$.  Furthermore, for $m=1,\ldots, 15$ we have $\bA_m^{-1}(C_2)=\emptyset$.
Therefore the assumptions of \Cref{prop:mainregext} are satisfied and we obtain
\begin{align*}
 \P(\bX \in tC_2)\sim \P(Z^{(2)}>t)  \bE_2^{(3)}[\mu_2(\bA^{-1}(C_2))] \sim\frac12 \kappa_1\kappa_2(q_{16}+q_{19})t^{-2\alpha}.
\end{align*}

\halmos
\end{example}
Examples for different choices of risk sets $C$ are aplenty considering the numerous risk situations for a group of agents.
In what follows, we address tail risks for extreme events where the portfolio risk of all agents are above a high threshold in a more systematic way. Such events are represented by sets of the form $t(\bx,\infty)$.  In case we want to study the problem for a specific set of agents, we need only to consider a reduced set of rows of the adjacency matrix $\bA$.

We suppose that each agent is able to take investment decisions according to a probability distribution, where the agents' choices are independent of each other.
Hence, we may assume for each agent $k\in\mathcal{A}=\{1,\ldots,q\}$ that there exist subsets $J_{k1},\ldots,J_{km_k}$ of  investments $\mathcal{O}=\{1,\ldots, d\}$ such that
\[\P(A_{k i}= 1 \text{ for }i\in J_{kl}, \text{and } A_{k i}=0 \text{ for } i\in J_{kl}^{c}) = p_{kl}\]
for $p_{kl}\in[0,1]$ and $\sum_{l=1}^{m_k}p_{kl}=1$. 
We may also consider $A_{ki}=W_{ki}>0$ for $i\in J_{kl}$ as in \eqref{Aarbitrary}.

Our examples also show exactly how our results extend those of Theorem~\ref{lem:jd2017thm23} of Janssen and Drees \cite{janssen:drees:2016} in multiple directions. Firstly, we allow for a non-square matrix $\bA$ with $q\ge d$, whereas  the result in \cite{janssen:drees:2016} was restricted to $q=d$. For computational ease we restrict to the case where the components of $\bZ$ are independent.
Although this results in $\bZ\in\MRV(d\al ,\mu_d,\E^{(d)}_{d})$ as required in \cite{janssen:drees:2016},
 we obtain $\bA\bZ$ to be multivariate regularly varying with different indices in different spaces; whereas in the aforementioned paper, the indices of regular variation of $\bZ$ and $\bA\bZ$ on $\E^{(d)}_{d}$ remain identical.

Our first result provides tail probabilities for the agents' risk exposures for a model, where each agent invests in exactly one investment possibility and the investment possibilities are independent of each other. {Moreover, agents take their investment decisions independently. }

To invest into one investment possibility is a risk averse investment strategy for small $\alpha$. According to Remark 13.3(b) of R\"uschendorf \cite{Ruschendorf}, for $\alpha\le 1$ portfolio diversification does not reduce the danger of extreme losses, but typically increases extreme risks.

\begin{proposition}\label{prop:onerisk1}
Let $Z_{1},\ldots, Z_{d}$ be independent random variables such that \linebreak
$\P(Z_{i}>t)\sim \kappa_{i}t^{-\alpha}$ for $\alpha>0$ as $t\to\infty$ with constants $\kappa_{i}>0$ for $i\in\mathcal{O}=\{1,\ldots, d\}$.
Let $\bA\in \{0,1\}^{q\times d}$ for $q\geq d$ be a random adjacency matrix, where for all $k \in \mathcal{A}=\{1,\ldots,q\}$ independently,

\begin{eqnarray*}
\begin{array}{rccll}
\P(\,A_{k i}=1 \text{ and } A_{k j}=0 \text{ for } j\neq i ) &=& \frac{1}{d-1}, &\quad i\in\{1,\ldots,d\}\backslash\{k\}, & k\in\{1,\ldots,d\},\\
\P(\,A_{k i}=1 \text{ and } A_{k j}=0 \text{ for } j\neq i ) &=& \frac{1}{d}, &\quad i\in\{1,\ldots,d\},& k\in\{d+1,\ldots,q\}.
\end{array}
\end{eqnarray*}
\begin{enumerate}[(a)]
\item
For $1\leq k\leq q-1$
we have $\bA\bZ\in\MRV(\alpha,b_1(t),\overline{\mu}_k,\E^{(k)}_{q})$ with
$$\overline{\mu}_k(\cdot)=\bE_1^{(k)}\left[ \mu_1(\{\bz\in \E^{(1)}_{d}:\bA\bz \in \,\cdot\,\})\right],$$
where $\mu_1(\left[0,\bz\right]^c)=K_1^{-1}\sum_{i=1}^d\kappa_i z_i^{-\alpha}$ for $\bz\in \E^{(1)}_{d}, b_1(t)\sim  (K_1t)^{1/\alpha}$ and $K_1=\sum_{i=1}^d \kappa_i$.
    \item
  We have $\P(\Omega_1^{(q)})=0$  and for $2\leq i\leq d$ and $\bx=(x_{1},\dots,x_{q})^\top\in\E_q^{(q)}$ we have as $t\to\infty$,
\begin{eqnarray*}
 \P_{i}^{(q)}(\bA\bZ \in t(\bx,\binfty) )   &=&  \Delta_i \sum\limits_{{1\le j_{1}<\cdots<j_{i}\le d}}  \left\{\prod\limits_{l=1}^{i } \kappa_{j_{l}}x_{j_l}\right\} t^{-i\alpha} + o(t^{-i \alpha}).
\end{eqnarray*}
where
\beao
  \Delta_i=\left(\frac{i-1}{d-1}\right)^{i}\left(\frac{i}{d-1}\right)^{d-i}
        \left(\frac{i}{d}\right)^{q-d}-i\left(\frac{i-2}{d-1}\right)^{i-1}
        \left(\frac{i-1}{d-1}\right)^{d+1-i}\left(\frac{i-1}{d}\right)^{q-d}.
\eeao
\item
For $k=q$ we have  $\bA\bZ\in\MRV(2\alpha,b_2(t),\overline{\mu}_q,\E^{(q)}_{q})$ with $b_2(t)\sim  (K_2t)^{1/2\alpha}$ where \linebreak $K_2=\sum_{1\le i<j\le d}^d \kappa_i\kappa_j$ and
\beao
    \overline{\mu}_q(\left(\bx,\binfty\right))=\frac{2^{q-2}}{(d-1)^dd^{q-d}}K_2^{-1}\sum_{1\leq i<j\leq d} \kappa_i\kappa_{j}(x_ix_j)^{-\alpha},
\eeao
such that for $\bx\in\E_q^{(q)}$ as $t\to\infty$,
\begin{align*}
\P(\bA\bZ \in t(\bx,\binfty) )   = & K_2\overline{\mu}_q(\left(\bx,\binfty\right)) t^{-2\alpha} + o(t^{-2\alpha}).
\end{align*}
\end{enumerate}
\end{proposition}

\begin{proof}
   First note that using similar arguments as in Lemma~\ref{ex:indpar}, we have for any $i=1,\dots,d$ that
    $\bZ\in \MRV(i\alpha,b_i,\mu_i,\E_d^{(i)})$ with canonical choices ${b_i(t) \sim (K_it)^{1/(i\alpha )}}$ and
    \beam\label{eq:Ki}
    K_i=\sum_{1\le j_{1}<\cdots<j_{i}\le d} \prod_{l=1}^{i} \kappa_{j_{l}}.
    \eeam
     Moreover,
for $\bz=(z_{1},\ldots,z_{d})^{\top} \in \E^{(d)}_d$ we have
 \begin{align} \label{limit:mui1}
\nonumber \mu_i(\{\bv\in \E^{(i)}_d: v_{j_{1}} > z_{j_{1}}, & \ldots, v_{j_{i}} > z_{j_{i}} \text{ for some }1\le  j_{1} <\cdots< j_{i}\le d\})\\ 
 & = \frac{1}{K_i}\sum_{1\le j_{1}<\cdots<j_{i}\le d} {\prod_{l=1}^{i} \kappa_{j_l} z_{j_{l}}^{-\alpha}},
 \end{align}
and
\begin{align*}
    \P(\tau_{d}^{(i)}(\bZ)>t )=\P(Z^{(i)}>t) {\sim 1/b_i^\leftarrow(t)} \sim K_it^{-i\alpha},\quad t\to \infty.
\end{align*}
The structure of $\bA$ guarantees that $\bE_i^{(q)}\big[\tau_{q,d}^{(k,i)}(\bA)^{i\alpha}\big]=1$, satisfying condition (i) of Theorems~\ref{thm:breiman} and \ref{prop:mainreg}.
 Also, referring to Remark~\ref{remarkCA} and \eqref{limit:mui1}, we have
$\mu_i(\widetilde\bCA_d^{(i)} (j))>0$ for all $j=1,\ldots, \binom{d}{i}$, satisfying condition (ii) in Theorems \ref{thm:breiman} and \ref{prop:mainreg}.
 Now we show the various parts of the result.
\begin{enumerate}[(a)]
\item
Let $1\leq k\leq q-1$. Then
$\P(\Omega_1^{(k)})>0$ and hence, the conclusion follows from Theorem~\ref{prop:mainreg}.
\item
Since each row of $\bA$ has exactly one entry 1 and all others zero, we have for $i=1,\ldots, d$,
\[\Omega_i^{(q)} =\{\omega\in \Omega: \text{exactly $i$ columns of $\bA_{\omega}$ {have at least one entry 1}} \},\]
because ${\tau^{(k,i)}_{q,d}}(\bA_\omega)<\infty$ if and only if  there are not more than $i$ columns of
$\bA_\omega$ with {positive} entries.  \\
Clearly $\Omega_1^{(q)}=\emptyset$ and we have $\P(\Omega_1^{(q)})=0$.

Now for $1\leq j_1<\ldots<j_i\leq d$ and $i=1,\ldots,d$ define
\beao
    \Omega_{j_1,\ldots,j_i}^{(q)} :=\{\omega\in \Omega: \text{exactly columns $j_1,\ldots,j_i$ of $\bA_{\omega}$ {have at least one entry 1}} \}.
\eeao
Hence for $2\leq i\leq d$,
\begin{align*}
    \P(\Omega_{j_1,\ldots,j_i}^{(q)}) &=\left(\frac{i-1}{d-1}\right)^{i}\left(\frac{i}{d-1}\right)^{d-i}
        \left(\frac{i}{d}\right)^{q-d}-i\left(\frac{i-2}{d-1}\right)^{i-1}
        \left(\frac{i-1}{d-1}\right)^{d+1-i}\left(\frac{i-1}{d}\right)^{q-d}\\
         &=:\Delta_i.
\end{align*}
Now an application of Theorem \ref{thm:breiman} yields
\beao
    \lefteqn{\P_{i}^{(q)}(\bA\bZ \in t(\bx,\binfty) )}\\
    &=&\P(\tau_{d}^{(i)}(\bZ)>t )\bE_i^{(k)}\big[ \mu_{i}(\{\bz\in\E_d^{(i)}:\bA\bz \in (\bx,\binfty)\})\big]+o\big(\P(\tau_{d}^{(i)}(\bZ)>t )\big)\\[2mm]
    &=&\sum_{1\le j_{1}<\cdots<j_{i}\le d}\mu_i\big(\bz\in\R^d_+:z_{j_1}>x_{j_1},\ldots,z_{j_i}>x_{j_i}\big) \P\big(\Omega_{j_1,\ldots,j_i}^{(q)}\big)K_it^{-i\alpha}+o(t^{-i\alpha})\\
    &=& \Delta_i \sum_{1\le j_{1}<\cdots<j_{i}\leq d}\prod_{l=1}^{i}\kappa_{j_l}x_{j_l}^{-\alpha} t^{-i\alpha}+o(t^{-i\alpha})
\eeao
which is the result in (b).
\item
 Using notation from Theorem \ref{prop:mainreg}, we have $i_d^*=2$. Also note that for $i<j$, $$\P(\Omega_{i,j}^{(q)})=\Delta_2=\frac{2^{q-2}}{(d-1)^qd^{q-d}}.$$ Therefore using Theorem~\ref{prop:mainreg}(b) we have $\bA\bZ\in\MRV(2\alpha,b_2(t),\overline{\mu}_q,\E^{(q)}_{q})$ and as $t\to\infty$,
\begin{eqnarray*}
    \P(\bA\bZ\in t(\bx,\infty))
    &=& \bE_2^{(q)} [\mu_2(\bA^{-1}((\bx,\binfty)))] \P(Z^{(2)}>t)+o(\P(Z^{(2)}>t))\\[2mm]
   &=& \Big\{\sum_{1\leq i<j\leq d} \mu_2(\bz\in \R_+^d:z_i>x_i,z_j>x_j)
        \P\left(\Omega_{i,j}^{(q)}\right)\Big\} \, K_2t^{-2\alpha}+o(t^{-2\alpha})\\
   & =&\Big\{\sum_{1\leq i<j\leq d} \kappa_i\kappa_{j}(x_ix_j)^{-\alpha}\Big\}\frac{2^{q-2}}{(d-1)^qd^{q-d}} t^{-2\alpha}+o(t^{-2\alpha})
\end{eqnarray*}
which shows (c).
    \end{enumerate}
\end{proof}

Proposition~\ref{prop:onerisk1} shows that for sets  $C \subset \E_q^{(k)}$ for $k\in\{1,\ldots, q-1\}$,  $\P(\bA\bZ \in tC)$ is of the order $t^{-\alpha}$. But for sets  of the form $t(\bx,\binfty)$ which belong to $\E_q^{(q)}$, we observe a tail probability of the order $t^{-2\alpha}$. However, if we restrict $\bA$ to $\Omega_i^{(k)}$ as in part (b), we may observe tail probabilities of the order $t^{-i\alpha}$ for all $i=2,\ldots,d$. 

In the next example we show that tail probabilities of other orders can also be observed for risk sets of the form $t(\bx,\binfty)$.
Here we fix $q=d$ and consider the same investment scenario as in Proposition~\ref{prop:onerisk1}; i.e., each agent invests in exactly one investment possibility and agents take their investment decisions independently.
As before each row is a unit vector, but the distribution of $\bA$ changes. 
Given $m\in \{1,\ldots, d-1\}$,  the single 1 in each row is  chosen uniformly on a subset of size $d-m$, the subset changing across each row.

From a mathematical point of view,  we obtain multivariate regular variation with different indices on $\E_d^{(d)}$ depending on the choice of $m$. Such a model leads to explicit expressions for the asymptotic tail probabilities for $\P(\bA\bZ\in t(\bx,\binfty))$.

\begin{proposition}\label{prop:onerisk2}
Let  $Z_{1},\ldots, Z_{d}$ be independent random variables such that \linebreak
$\P(Z_{i}>t)\sim \kappa_{i}t^{-\alpha}$ {for $\alpha>0$} as $t\to\infty$ with constants $\kappa_{i}>0$ for $i\in\mathcal{O}=\{1,\ldots, d\}$.
Let {$1\le m\le d-1$, $m\in\N$} and $\bA\in \{0,1\}^{d\times d}$ be a random adjacency matrix, where for all $k \in \mathcal{A}=\{1,\ldots,d\}$ independently
\[\P(\,A_{k i}=1 \text{ and } A_{k j}=0 \text{ for } j\neq i ) = \frac{1}{d-m}, \quad i\in I_k,
\]
where $I_k$ is defined as
\beao
    I_k=\left\{\begin{array}{ll}
        \{1,\ldots,d\}\backslash\{k,\ldots,k+m-1\} & \text{ if } k+m-1\leq d,\\
        \{1,\ldots,d\}\backslash\{\{k,\ldots,d\}\cup\{1,\ldots,k+m-1-d\}\}     & \text{ if } k+m-1> d.
    \end{array}
    \right.
\eeao
Also define for $m+1\leq i \leq d$,
\beao
    q_i^{(m)}=\frac{i^{d-(i+m-1)}(i-m)^{i-m+1}\prod_{l=1}^{m-1} (i-l)^2}{(d-m)^d}.
\eeao
Then the following assertions hold:
\begin{enumerate}[(a)]
  \item
For $1\leq k\leq d-m$
we have $\bA\bZ\in\MRV(\alpha,b_1,\overline{\mu}_k,\E^{(k)}_{d})$ with
$$ \overline{\mu}_k(\cdot)=\bE_1^{(k)}\left[ \mu_1(\{\bz\in \E^{(1)}_{d}:\bA\bz \in \,\cdot\,\})\right],$$
and for $d-m<k\leq d$ we have
$\bA\bZ\in\MRV(j\alpha,b_j,\overline{\mu}_{k},\E^{(k)}_{d})$
with $$\overline{\mu}_k(\cdot)=\bE^{(k)}_{j}\left[ \mu_{j}(\{\bz\in \E^{(j)}_{d}:\bA\bz \in \,\cdot\,\})\right],$$
where $j=k+m+1-d$, $\mu_j$ is defined as in \eqref{limit:mui1}, and $b_j(t)\sim(K_jt)^{1/j\alpha}$ with $K_j$ as in \eqref{eq:Ki}.
\item For  $m+1\leq i \leq d$, and $\bx\in \E_d^{(d)}$ we have as $t\to\infty$,
\begin{eqnarray*}
 \P_{i}^{(d)}(\bA\bZ \in t(\bx,\binfty) )   &=& \left\{\sum\limits_{{1\le j_{1}<\cdots<j_{i}\le d}}  \prod\limits_{l=1}^{i } \kappa_{j_{l}}x_{j_l}^{-\alpha}\right\} \left\{q_i^{(m)}-q_{i-1}^{(m)}\right\} t^{-i \alpha} + o(t^{-i \alpha}).
\end{eqnarray*}
\item
For $k=d$, part (a) applies with
\beao
    \overline{\mu}_d(\left(\bx,\binfty\right))=K_{m+1}^{-1}q_{m+1}^{(m)}\left\{\sum\limits_{{1\le j_{1}<\cdots<j_{i}\le d}}  \prod\limits_{l=1}^{i } \kappa_{j_{l}}x_{j_l}\right\}, \quad \bx\in\E_d^{(d)},
\eeao
with $K_{m+1}$ as in \eqref{eq:Ki}, and we have as $t\to\infty$,
\begin{align*}
\P(\bA\bZ \in t(\bx,\binfty) )   = & K_{m+1}\overline{\mu}_d(\left(\bx,\binfty\right)) t^{-(m+1)\alpha} + o(t^{-(m+1)\alpha}).
\end{align*}
\end{enumerate}
\end{proposition}

\begin{proof}
For $m\leq i+1\leq d$,
\beao
    q_i^{(m)}=\P(\{\omega\in \Omega: \text{only in columns $j_1,\ldots,j_i$ of $\bA_{\omega}$ appears 1} \})
\eeao
such that $\P(\Omega_{j_1,\ldots,j_i}^{(d)})=q_i^{(m)}-q_{i-1}^{(m)}$ and $\P(\Omega_{j_1,\ldots,j_{m+1}}^{(d)})=q_{m+1}^{(m)}$.
The proposition  can then be proved in a similar manner as Proposition~\ref{prop:onerisk1}, and is omitted here.
\end{proof}

\subsection{Dependent objects}\label{ss43}


In this section we contrast independent objects as we have considered previously with a specific dependence structure of the components of $\bZ$. Moreover, we also investigate the influence of weights in a numerical example.
Let $\bX=\bA\bZ$ be the investment portfolios of five agents, each of whom connects to a subset of three objects whose risks are given by $\bZ$.
We estimate the tail risks for $k=1,\ldots,5$:
\begin{align*}
&\P(\text{risk of at least $k$ of the portfolios}>t) =: \P(\bX\in tD_k).
\end{align*}
We use a weighted adjacency matrix, which is now, however, deterministic and in both examples given by
\begin{align*}
\bA = \begin{bmatrix}
       W_{11} & 0 & 0          \\[0.3em]
       0 & W_{22} & 0         \\[0.3em]
       0 & 0 & W_{33}          \\[0.3em]
       W_{41} & W_{42} & 0 \\[0.3em]
       0 & W_{52} & W_{53}
     \end{bmatrix}
 \end{align*}
with weights $W_{11},W_{22},  W_{33},  W_{41},  W_{42},  W_{52},  W_{53}>0$.
Also for the convenience of computing the limit measures of the sets $D_k$ we assume $ W_{41} W_{11}^{-1}+ W_{42} W_{22}^{-1}>1$ and $ W_{52} W_{22}^{-1}+ W_{53} W_{33}^{-1}>1$.

Moreover, we assume that $\bZ=(Z_1,Z_2,Z_3)^{\top}$ has a probability distribution given by
\beam\label{depdis}
\P(Z_1\le z_1, Z_2 \le z_2, Z_3\le z_3) = (1+\theta(\kappa_1\kappa_2\kappa_3)^{\rho}(z_1z_2z_3)^{-\rho\alpha})\prod_{i=1}^3(1-\kappa_iz_i^{-\alpha}),
\eeam
for $ z_i\ge \kappa_i^{1/\alpha}$, where $\kappa_i>0, i=1,2,3,\; \alpha>0,\rho\ge 1, 0\le \theta\le 1$.
For $\theta=0$ the components of $\bZ$ are independent Pareto (cf. Example~\ref{ex:det}) and for $\rho=1, \theta=1$ they are dependent (cf. Example~\ref{ex:nonind}).
Such dependence in terms of copulas has been discussed in \citet{rodriguesetal:2004}.

This setting implies that in the two examples below, the underlying distribution of $\bZ$ has either independent marginals or at least it has a tractable form; the adjacency matrix $\bA$ is relatively simple, in order to provide an interpretable illustration.

\begin{example}\label{ex:det}
Suppose $Z_{1}, Z_{2}, Z_{3}$ are independent random variables such that \linebreak $\P(Z_{i}>z)= \kappa_{i}z^{-\alpha}, z>\kappa_{i}^{1/\alpha}$ with constants $\kappa_{i}>0$ for $i=1,2,3$.
\begin{figure}[t!]
    \centering
    \includegraphics[width=0.44\textwidth]{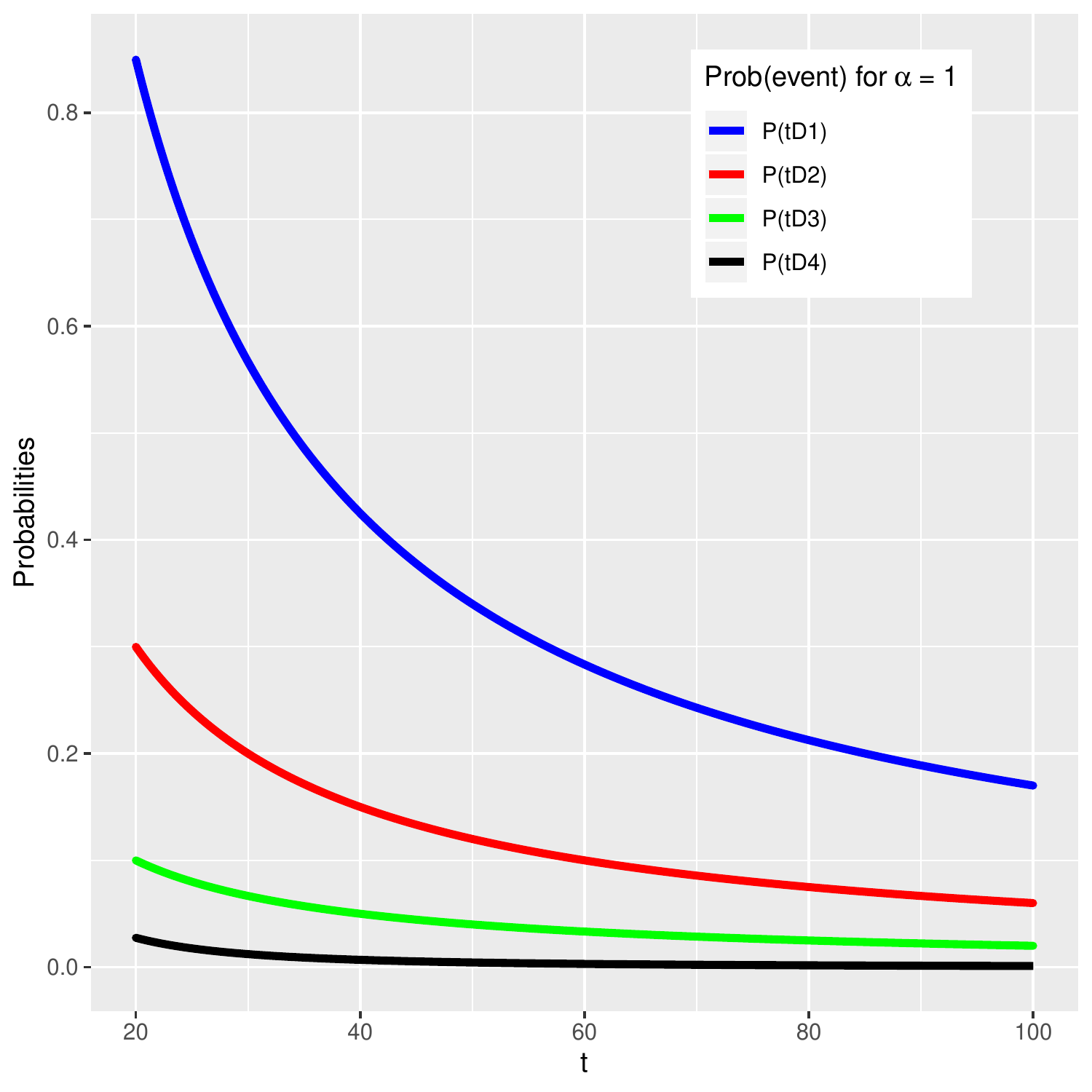}
    \includegraphics[width=0.44\textwidth]{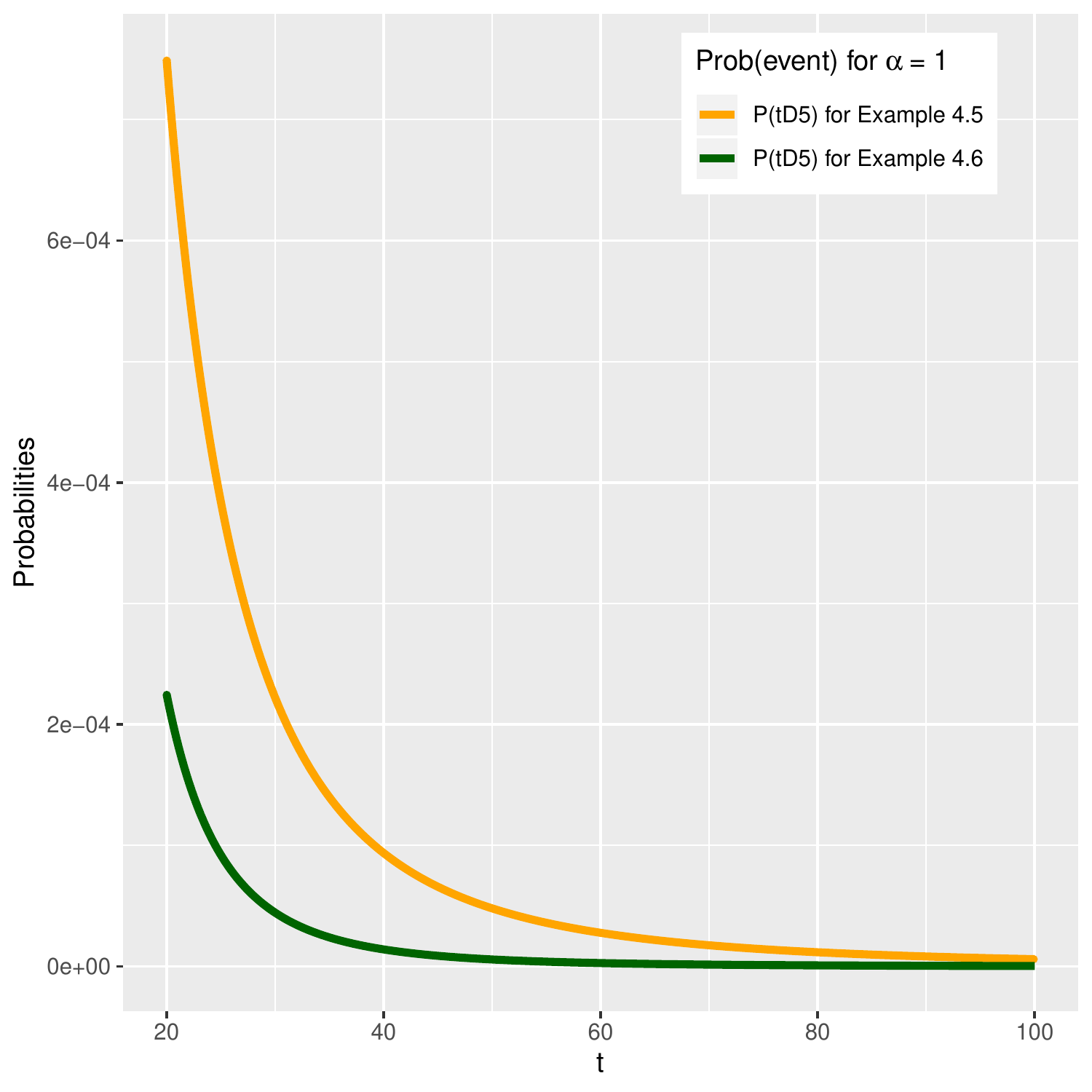}
    \includegraphics[width=0.44\textwidth]{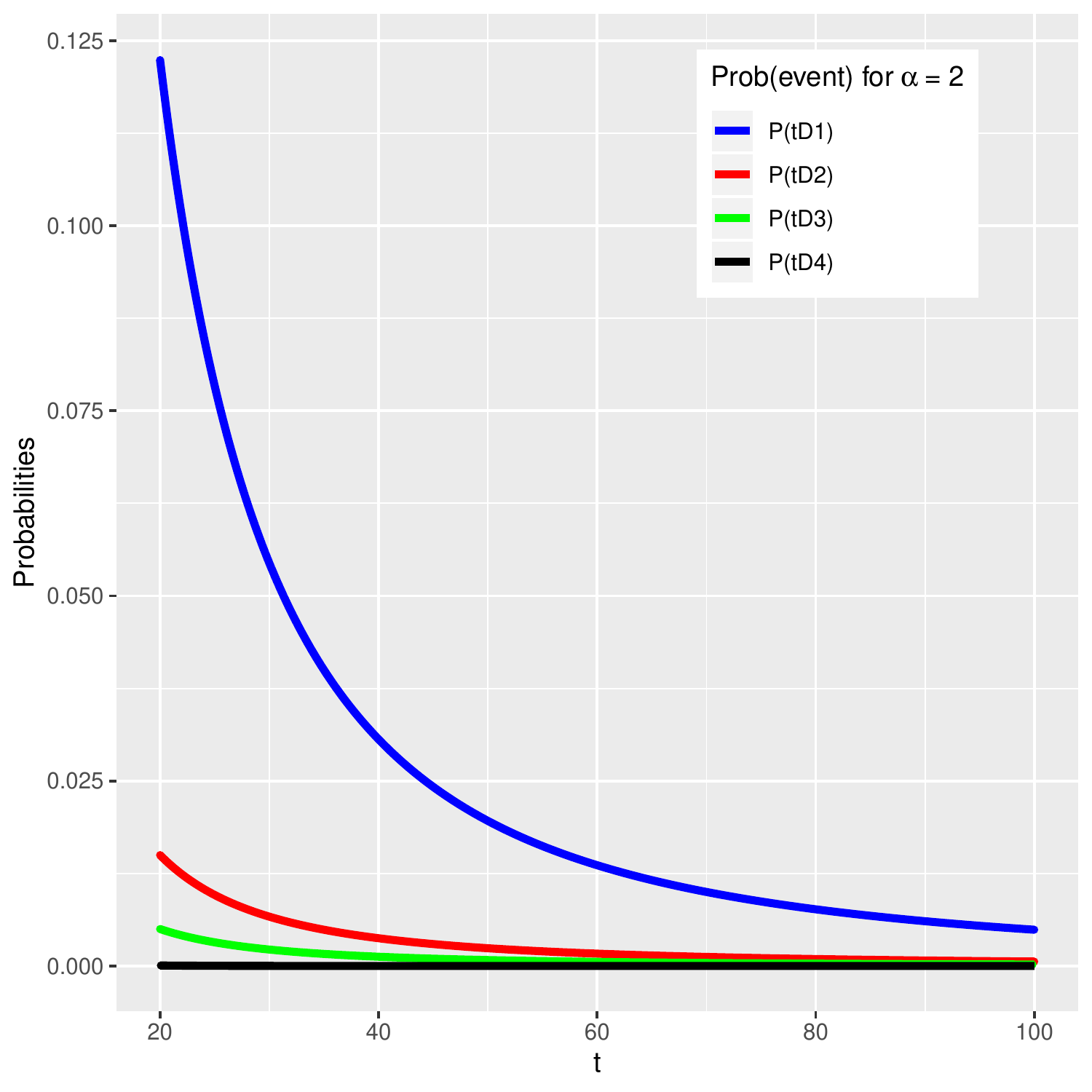}
    \includegraphics[width=0.44\textwidth]{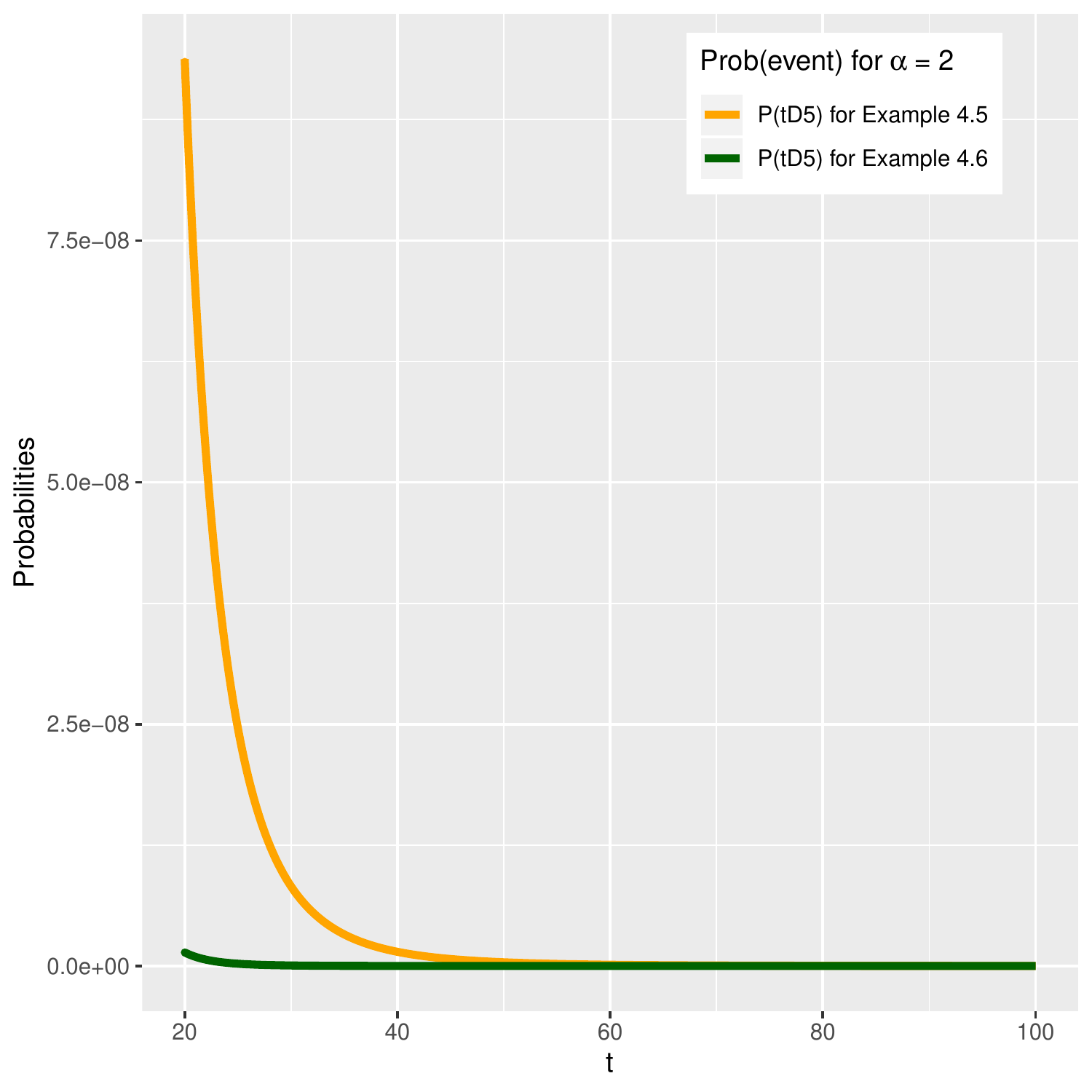}
\caption{The probabilities of the tail events $tD_1, tD_2, tD_3, tD_4$ are asymptotically equal in both Examples~\ref{ex:det} and~\ref{ex:nonind} and are plotted in the left panels with $\alpha=1$ in the top and $\alpha=2$ in the bottom. The probabilities of the tail events $tD_5$ are asymptotically different in the two examples and are plotted in the right panels ($\alpha=1$ in the top and $\alpha=2$ in the bottom). Values are plotted for $20\le t\le 100$.}
    \label{fig:plotprob1}
\end{figure}
We calculate all relevant quantities.
First, the tails of the order statistics are given for $t\to\infty$,
\begin{align}\begin{split}\label{eq:Pzi}
    \P(Z^{(1)}>t) &= (\kappa_1+\kappa_2+\kappa_3)t^{-\alpha} + o(t^{-\alpha}),\\
     \P(Z^{(2)}>t) &= (\kappa_1\kappa_2+\kappa_2\kappa_3+\kappa_3\kappa_1)t^{-2\alpha} + o(t^{-2\alpha}),\\
      \P(Z^{(3)}>t) &= (\kappa_1\kappa_2\kappa_3)t^{-3\alpha}.
      \end{split}
\end{align}
We have $\bZ\in \MRV(i\alpha,b_i,\mu_i,\E_3^{(i)})$ with canonical $b_i$ given by
\begin{align}
b_1(t) & = (\kappa_1+\kappa_2+\kappa_3)^{1/\alpha}t^{1/\alpha},\nonumber\\
 b_2(t) &= (\kappa_1\kappa_2+\kappa_2\kappa_3+\kappa_3\kappa_1)^{1/2\alpha}t^{1/2\alpha}, \label{def:b1b2}\\
b_3(t) & = (\kappa_1\kappa_2\kappa_3)^{1/3\alpha}t^{1/3\alpha}, \nonumber
\end{align}
and limit measures
\begin{align}
\begin{split}\label{eq:mui}
   & \mu_1\Big(\bigcup\limits_{i=1}^3 \left\{\bv\in \R_+^3: v_i>z_i\right\}\Big) =  ({\kappa_1+\kappa_2+\kappa_3})^{-1} \sum_{i=1}^3 \kappa_iz_i^{-\alpha},\\
    & \mu_2\Big(\bigcup\limits_{1\le i\neq j\le 3}\left\{\bv\in \R_+^3: v_i>z_i, v_j>z_j\right\}\Big) =  ({\kappa_1\kappa_2+\kappa_2\kappa_3+\kappa_3\kappa_1})^{-1}\sum\limits_{1\le i\neq j \le 3} \kappa_i\kappa_j(z_iz_j)^{-\alpha},\\
     & \mu_3\left((z_1,\infty)\times(z_2,\infty)\times(z_3,\infty)\right) =  (z_1z_2z_3)^{-\alpha}.
\end{split}
\end{align}
Note that $D_k\subset\E_5^{(k)}$ for $k=1,\ldots,5$. We can check from the form of $\bA$ that
\[i_1^*=1,\quad i_2^*=1, \quad i_3^*=1,\quad i_4^*=2, \quad i_5^*=3.\]
Hence using Theorem~\ref{prop:mainreg}, along with \eqref{eq:Pzi} and \eqref{eq:mui}, we have as $t\to\infty$,
\begin{align*}
 \P(\bA\bZ\in tD_1) &\sim \P(Z^{(1)}>t)\mu_1(\bA^{-1}(D_1)) \\
 &\sim \left[\kappa_1(\max( W_{11}^{\alpha}, W_{41}^{\alpha}))+\kappa_2(\max( W_{22}^{\alpha}, W_{42}^{\alpha}, W_{52}^{\alpha}))+\kappa_3(\max( W_{33}^{\alpha}, W_{53}^{\alpha}))\right]\, t^{-\alpha}.
 \end{align*}
 Similarly, we can show that as $t\to\infty$,
 \begin{align*}
 \P(\bA\bZ\in tD_2) & \sim \P(Z^{(1)}>t)\mu_1(\bA^{-1}(D_2)) \\
  &\sim \left[\kappa_1(\min( W_{11}^{\alpha}, W_{41}^{\alpha}))+\kappa_2(\min( W_{22}^{\alpha}, W_{42}^{\alpha}, W_{52}^{\alpha}))+\kappa_3(\min( W_{33}^{\alpha}, W_{53}^{\alpha}))\right]\, t^{-\alpha}.\\
 \P(\bA\bZ\in tD_3) & \sim \P(Z^{(1)}>t)\mu_1(\bA^{-1}(D_3))\\
                 &\sim \kappa_2\min( W_{22}^{\alpha}, W_{42}^{\alpha}, W_{52}^{\alpha})\, t^{-\alpha},\\
 \P(\bA\bZ\in tD_4) & \sim \P(Z^{(2)}>t)\mu_2(\bA^{-1}(D_4))\\
                 & \sim \left[\kappa_1\kappa_2 {W_{11}^{\alpha}}\min( W_{22}^{\alpha}, W_{52}^{\alpha})+ \kappa_2\kappa_3\min( W_{22}^{\alpha}, W_{42}^{\alpha}) W_{33}^{\alpha}\right.\\ & \quad\quad \quad\quad \quad\quad + \left.\kappa_3\kappa_1\min( W_{33}^{\alpha}, W_{53}^{\alpha})\min( W_{11}^{\alpha}, W_{41}^{\alpha})\right]\, t^{-2\alpha},\\
 \P(\bA\bZ\in tD_5) & \sim \P(Z^{(3)}>t)\mu_3(\bA^{-1}(D_5)) \sim \kappa_1\kappa_2\kappa_3 W_{11}^{\alpha}  W_{22}^{\alpha}  W_{33}^{\alpha} t^{-3\alpha}.
\end{align*}
The forms for $\P(\bA\bZ\in tD_4)$ and $\P(\bA\bZ\in tD_5)$ become more complicated if we do not assume $ W_{11}, W_{22},  W_{33},  W_{41},  W_{42},  W_{52},  W_{53}>0$. Furthermore, if all weights are equal to one, the above formulas hold true.

As an illustration, we fix $\kappa_1=1,\kappa_2=2,\kappa_3=3$.
Moreover, let $ W_{11}= W_{22}= W_{33}=1$, $ W_{41}=2,  W_{42}=2$, and $ W_{52}= W_{53}=3$. The five probabilities obtained above are plotted for the case $\alpha=1, 2$ in Figure \ref{fig:plotprob1} for $20\le t \le 100$.
We also compare the probability for the event $tD_5$ in this example and in Example \ref{ex:nonind}; see the right two panels of Figure \ref{fig:plotprob1}.
\halmos
\end{example}

\begin{example}\label{ex:nonind}
Suppose $Z_1,Z_2,Z_3$ are dependent with joint distribution \eqref{depdis} for $\rho=\theta=1$.
Otherwise assume the setting as in Example~\ref{ex:det}
We calculate again all relevant quantities. First, the tails of the order statistics are given for $t\to\infty$,
\begin{align*}\begin{split}
    \P(Z^{(1)}>t) &= (\kappa_1+\kappa_2+\kappa_3)t^{-\alpha} + o(t^{-\alpha}),\\
     \P(Z^{(2)}>t) &= (\kappa_1\kappa_2+\kappa_2\kappa_3+\kappa_3\kappa_1)t^{-2\alpha} + o(t^{-2\alpha}),\\
      \P(Z^{(3)}>t) &= \kappa_1\kappa_2\kappa_3(\kappa_1+\kappa_2+\kappa_3)t^{-4\alpha} + o(t^{-4\alpha}).
      \end{split}
\end{align*}
Notice that the only difference from \eqref{eq:Pzi} is in the term $\P(Z^{(3)}>t)$.
Hence,  for $i=1,2$ we have $\bZ\in \MRV(i\alpha,b_i,\mu_i,\E_3^{(i)})$ with canonical choice for $b_i$ as in \eqref{def:b1b2} and $\mu_i$ as in \eqref{eq:mui}.
 On the other hand, we have $\bZ\in \MRV(4\alpha,b_3,\mu_3,\E_3^{(3)})$ with
\[b_3(t)= (\kappa_1\kappa_2\kappa_3(\kappa_1+\kappa_2+\kappa_3))^{1/(4\alpha)}t^{1/(4\alpha)},\]
and
\begin{align*}
\begin{split}
     & \mu_3\left((z_1,\infty)\times(z_2,\infty)\times(z_3,\infty)\right) = ({\kappa_1+\kappa_2+\kappa_3})^{-1} \sum_{i=1}^3 \kappa_iz_i^{-\alpha} (z_1z_2z_3)^{-\alpha}.
\end{split}
\end{align*}
As in Example \ref{ex:det}, we have $i_1^*=1,i_2^*=1,i_3^*=1,i_4^*=2,i_5^*=3.$
Using Theorem~\ref{prop:mainreg}, along with \eqref{eq:Pzi} and \eqref{eq:mui}, we have the same limits for $\P(\bA\bZ\in tD_k)$ for $k=1,\ldots,4$. The only difference occurs for $k=5$, where we have for $t\to\infty$,
\begin{align*}
 \P(\bA\bZ\in tD_5) & \sim \P(Z^{(3)}>t)\mu_3(\bA^{-1}(D_5))\\
                               &  \sim \left[\kappa_1\kappa_2\kappa_3 W_{11}^{\alpha}  W_{22}^{\alpha}  W_{33}^{\alpha} (\kappa_1 W_{11}^{\alpha}+\kappa_2 W_{22}^{\alpha}+\kappa_3 W_{33}^{\alpha})\right]t^{-4\alpha}.
\end{align*}
{Again we fix $\kappa_1=1,\kappa_2=2,\kappa_3=3$ and  let $ W_{11}= W_{22}= W_{33}=1$, $ W_{41}=2,  W_{42}=2$, and $ W_{52}=W_{53}=3$, as in Example \ref{ex:det}. The  probabilities for events $tD_1, tD_2, tD_3, tD_4$ asymptotically remain the same as  in Example \ref{ex:det} (matching the plots in the left two panels of Figure \ref{fig:plotprob1} for the case $\alpha=1, 2$). 
In the right panels Figure \ref{fig:plotprob1} we plot the values for $tD_5$ when $\alpha=1, 2$; clearly these values differ in the two examples.}
\halmos
\end{example}

\section{Conclusion}\label{sec:concl}

This work is motivated by the need to find probabilities of a variety of extreme events under a linear transformation of regularly varying random vectors.  By an extension of Breiman's Theorem we have shown that  probabilities of many such events can be calculated, if we have information on the regular variation property of the underlying random vector on specific subcones of the positive quadrant. Most of the subsets $C$ of such cones have linear boundaries and hence form a polytope, whose pre-image under linear transformation also turns out to be a polytope in $\R_+^d$. Computing the limit measures $\bE_i^{(k)}[\mu_i(\bA^{-1}(C))]$ in such cases means finding the appropriate boundaries of the polytope which can become quite complicated.  For moderate dimensions of the matrix $\bA$, numerical solutions can be obtained even when the distributional forms of $\bZ$ and $\bA$ are more complicated.

We envisage wide application of such results in  areas of risk management. There are clear implications  for computing conditional value at risk, as well as a variety of conditional risk measures. We also believe that an alternative characterization of the rate of decay of tail probabilities can be provided via connectivity of the row components (in the bipartite network model, the agents); this work is under current investigation.

\bibliographystyle{imsart-nameyear}
\bibliography{bibfilenew}

\end{document}